\documentclass[preprint]{elsarticle}

\usepackage{graphicx}
\usepackage{amsmath,amssymb} 
\usepackage{color}
\usepackage[mathscr]{euscript}
\usepackage[font=footnotesize]{subfig}
\usepackage{multirow}
\usepackage{tikz}
\usepackage{circuitikz}

\usepackage{enumitem}
\usepackage{amsthm,comment}
\usepackage{microtype,hyperref}
\usepackage{algorithm2e}
\usepackage{floatrow}
\usepackage{graph}
\newtheorem{theorem}{Theorem}

\newtheorem{proposition}{Proposition}
\newtheorem{lemma}{Lemma}
\newtheorem{claim}{Claim}

\newtheorem{question}{Question}
\newcounter{example2}
\newtheorem{observation}[example2]{Observation}

\begin{document}

\begin{frontmatter}



\title{Triangle-free projective-planar graphs with diameter two: domination and characterization \footnote{This work is partially supported by the IFCAM project Applications of graph homomorphisms (MA/IFCAM/18/39)}}
 \author[label1]{{Dibyayan Chakraborty} }
 \author[label2]{{Sandip Das}}
  \author[label2]{{Srijit Mukherjee}}
   \author[label2]{{Uma kant Sahoo}}
 \author[label3]{{Sagnik Sen}}
 
 \address[label1] {Indian Institute of Science, Bengaluru}
 \address[label2]{Indian Statistical Institute, Kolkata, India}
 \address[label3]{Indian Institute of Technology Dharwad, India}

\author{}

\address{}

\begin{abstract}
In 1975, Plesn\'ik characterized all triangle-free planar graphs as having a diameter $2$. We characterize all triangle-free projective-planar graphs having a diameter $2$ and discuss some applications. In particular, the main result is applied to calculate the analogue of clique numbers for graphs, namely, colored mixed graphs, having different types of arcs and edges. 
\end{abstract}


\end{frontmatter}



\section{Introduction and main results}
In 1975, Plesn\'ik~\cite{plesnik1975} characterized all triangle-free planar graphs\footnote{In this article, we use the notation and terminology of ``Introduction to Graph Theory'' by D. B. West~\cite{west}.} having diameter $2$ by proving the following result.

\begin{theorem}[Plesn\'ik 1975~\cite{plesnik1975}]\label{th plesnik}
A triangle-free planar graph $G$ has a diameter $2$ if and only if it is isomorphic to one of the following graphs:
\begin{enumerate}
    \item[(i)] $K_{1,n}$ for $n \geq 2$,
    
    \item[(ii)] $K_{2,n}$ for  $n \geq 2$,
    
    \item[(iii)] the graph $C_5(m,n)$ obtained by adding $(m+n)$ degree-$2$
    vertices to the $5$-cycle $C_5 = v_1v_2v_3v_4v_5v_1$,
    for  $m,n \geq 0$,
    in such a way that $m$ of the vertices are adjacent  to $v_1, v_3$ and $n$ of the vertices are adjacent to $v_1, v_4$. 
\end{enumerate}
\end{theorem}

We prove the analogue of Plesn\'ik's result for \textit{projective-planar graphs}, that is, graphs that can be embedded on the non-orientable surface of Euler genus one (also known as the real projective plane) without their edges crossing each other except, maybe,  on the vertices. 
For convenience, let us refer to the
graphs listed in Theorem~\ref{th plesnik} as 
\textit{Plesn\'ik graphs}.

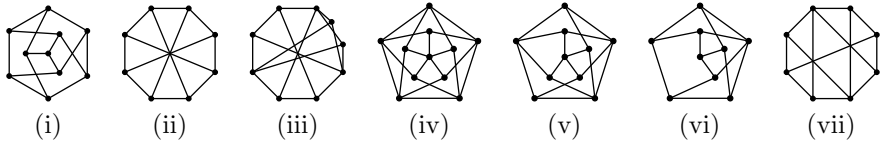
\begin{figure}[ht]
	\centering
	\begin{tabular}{ccccccc} 
    \begin{tikzpicture}[scale=0.3]
	\def\A{0,2}
    \def\B{1.73,1}
    \def\C{1.73,-1}
    \def\D{0,-2}
    \def\E{-1.73,-1}
    \def\F{-1.73,1}
    \def\G{0.5,0.86}
    \def\H{0.5,-0.86}
    \def\I{-1,0}
    \def\J{0,0}
    \renewcommand{\vertexset}{(a,\A,black),(b,\B,black),(c,\C,black),(d,\D,black),(e,\E,black),(f,\F,black),(g,\G,black),(h,\H,black),(i,\I,black),(j,\J,black)};
   \renewcommand{\edgeset}{(a,b),(b,c), (c,d),(d,e),(e,f),(f,a),(a,i),(d,i),(f,g),(g,c),(e,h),(b,h),(i,j),(g,j),(h,j)}
    \renewcommand{\defradius}{.1}
	\drawgraph

\end{tikzpicture} &
\begin{tikzpicture}[scale=0.3]
	\def\A{0.83,2}
    \def\B{2,0.83}
    \def\C{2,-0.83}
    \def\D{0.83,-2}
    \def\E{-0.83,-2}
    \def\F{-2,-0.83}
    \def\G{-2,0.83}
    \def\H{-0.83,2}
    
    \renewcommand{\vertexset}{(a,\A,black),(b,\B,black),(c,\C,black),(d,\D,black),(e,\E,black),(f,\F,black),(g,\G,black),(h,\H,black)};
   \renewcommand{\edgeset}{(a,b),(b,c), (c,d),(d,e),(e,f),(f,g),(g,h),(h,a),(a,e),(b,f),(c,g),(d,h)}
    \renewcommand{\defradius}{.1}
	\drawgraph

\end{tikzpicture} &
\begin{tikzpicture}[scale=0.3]
	\def\A{0.83,2}
    \def\B{1.5,1.4}
    \def\C{2,-0.83}
    \def\D{0.83,-2}
    \def\E{-0.83,-2}
    \def\F{-2,-0.83}
    \def\G{-2,0.83}
    \def\H{-0.83,2}
    \def\I{2,0.4}
    
    \renewcommand{\vertexset}{(a,\A,black),(b,\B,black),(c,\C,black),(d,\D,black),(e,\E,black),(f,\F,black),(g,\G,black),(h,\H,black),(i,\I,black)};
   \renewcommand{\edgeset}{(a,b),(b,c), (c,d),(d,e),(e,f),(f,g),(g,h),(h,a),(a,e),(b,f),(c,g),(d,h),(i,f),(i,a),(i,c)}
    \renewcommand{\defradius}{.1}
	\drawgraph

\end{tikzpicture} &
\begin{tikzpicture}[scale=0.34]
	\def\A{0,2}
    \def\B{1.9,0.62}
    \def\C{1.18,-1.62}
    \def\D{-1.18,-1.62}
    \def\E{-1.9,0.62}
    \def\F{0,1}
    \def\G{0.95,0.31}
    \def\H{0.59,-0.81}
    \def\I{-0.59,-0.81}
    \def\J{-0.95,0.31}
    \def\K{0,0}
    \renewcommand{\vertexset}{(a,\A,black),(b,\B,black),(c,\C,black),(d,\D,black),(e,\E,black),(f,\F,black),(g,\G,black),(h,\H,black),(i,\I,black),(j,\J,black),(k,\K,black)};
   \renewcommand{\edgeset}{(a,b),(b,c), (c,d),(d,e),(e,a),(a,g),(a,j),(b,f),(b,h),(c,g),(c,i),(d,h),(d,j),(e,i),(e,f),(k,f),(k,g),(k,h),(k,i),(k,j)}
    \renewcommand{\defradius}{.1}
	\drawgraph

\end{tikzpicture} &
\begin{tikzpicture}[scale=0.34]
	\def\A{0,2}
    \def\B{1.9,0.62}
    \def\C{1.18,-1.62}
    \def\D{-1.18,-1.62}
    \def\E{-1.9,0.62}
    \def\F{0,1}
    \def\G{0.95,0.31}
    \def\H{0.59,-0.81}
    \def\I{-0.59,-0.81}
    \def\K{0,0}
    \renewcommand{\vertexset}{(a,\A,black),(b,\B,black),(c,\C,black),(d,\D,black),(e,\E,black),(f,\F,black),(g,\G,black),(h,\H,black),(i,\I,black),(k,\K,black)};
   \renewcommand{\edgeset}{(a,b),(b,c), (c,d),(d,e),(e,a),(a,g),(b,f),(b,h),(c,g),(c,i),(d,h),(e,i),(e,f),(k,f),(k,g),(k,h),(k,i)}
    \renewcommand{\defradius}{.1}
	\drawgraph

\end{tikzpicture} &
\begin{tikzpicture}[scale=0.34]
	\def\A{0,2}
    \def\B{1.9,0.62}
    \def\C{1.18,-1.62}
    \def\D{-1.18,-1.62}
    \def\E{-1.9,0.62}
    \def\F{0,1}
    \def\G{0.95,0.31}
    \def\H{0.59,-0.81}
    \def\K{0,0}
    \renewcommand{\vertexset}{(a,\A,black),(b,\B,black),(c,\C,black),(d,\D,black),(e,\E,black),(f,\F,black),(g,\G,black),(h,\H,black),(k,\K,black)};
   \renewcommand{\edgeset}{(a,b),(b,c), (c,d),(d,e),(e,a),(a,g),(b,f),(b,h),(c,g),(d,h),(e,f),(k,f),(k,g),(k,h)}
    \renewcommand{\defradius}{.1}
	\drawgraph

\end{tikzpicture}&
\begin{tikzpicture}[scale=0.3]
	\def\A{0.83,2}
    \def\B{2,0.83}
    \def\C{2,-0.83}
    \def\D{0.83,-2}
    \def\E{-0.83,-2}
    \def\F{-2,-0.83}
    \def\G{-2,0.83}
    \def\H{-0.83,2}
    
    \renewcommand{\vertexset}{(a,\A,black),(b,\B,black),(c,\C,black),(d,\D,black),(e,\E,black),(f,\F,black),(g,\G,black),(h,\H,black)};
   \renewcommand{\edgeset}{(a,b),(b,c), (c,d),(d,e),(e,f),(f,g),(g,h),(h,a),(a,d),(b,f),(c,h),(d,g),(e,h)}
    \renewcommand{\defradius}{.1}
	\drawgraph

\end{tikzpicture}\\
(i) & (ii) & (iii) & (iv) & (v) & (vi) & (vii)
    \end{tabular}
    \caption{(i) The  Petersen graph $P_{10}$, 
    (ii) The Wagner graph $W_{8}$, 
    (iii) The graph $W_{8}^+$,
    (iv) The Gr\"otzsch graph $M_{11}$, 
    (v) The graph $M^-_{11}$, 
    (vi) The graph $M^=_{11}$, 
    (vii) The graph $K^*_{3,4}$.} 
    \label{fig:family-D}
\end{figure}

\begin{theorem}\label{th t-free ppg}
A triangle-free projective-planar graph $G$ has diameter $2$ if and only if it is isomorphic to one of the following:
\begin{enumerate}
    \item[(i)] a Plesn\'ik graph,
    
    \item[(ii)] $K_{3,3}$ or $K_{3,4}$,
    
    \item[(iii)] the graph $K_{3,3}(n)$ obtained by adding 
    $(n-1)$ parallel edges $e_2, e_3, \ldots, e_n$ to one of the edges $e_1$ of $K_{3,3}$  and subdividing each $e_i$ exactly once for $n \geq 1$,
    
    \item[(iv)] the graph $K_{3,4}(n)$ obtained by adding 
    $(n-1)$ parallel edges $e_2, e_3, \ldots, e_n$ to one of the edges $e_1$ of $K_{3,4}$  and subdividing each $e_i$ exactly once for $n \geq 1$,

    \item[(v)] one of the seven graphs depicted in Figure~\ref{fig:family-D}.
\end{enumerate}
\end{theorem}

Let us now discuss a few more results regarding properties of graphs having small diameters that can be embedded on a given surface $\mathbb{S}$ to place our work into context. Since $\mathbb{S}$ has an Euler characteristic, it follows from the Euler's formula~\cite{mohar2001} that any graph embedded in $\mathbb{S}$ has a bounded minimum degree. This implies that if we consider such a graph with diameter 2, then its domination number is at most its minimum degree.

\medskip

In 1996, MacGillivray and Seyffarth~\cite{macgillivray1996} proved that planar graphs with diameter $2$ have domination numbers at most $3$. In 2002, Goddard and Henning~\cite{goddard2002} showed that there is exactly one planar graph  having diameter $2$ 
that has a domination number equal to $3$. 
They also proved that for each surface (orientable or non-orientable) $\mathbb{S}$, there are finitely many graphs having diameter $2$ and domination number at least $3$ that can be embedded in $\mathbb{S}$.
A natural question to ask  in this context is the following. 

\begin{question}\label{Q g}
Given a surface $\mathbb{S}$, can you find the list  of all 
graphs having diameter $2$ and domination number at least $3$ 
that can be embedded on $\mathbb{S}$?
\end{question}

 As we just mentioned,
Goddard and Henning~\cite{goddard2002} answered 
Question~\ref{Q g} when $\mathbb{S}$ is the sphere (or equivalently, the Euclidean plane).
However, it seems that the question can be very difficult to answer in general as the tight upper bounds on the domination number for a family of graphs that can be embedded on a surface, other than the sphere, is yet to be found. 
Therefore, it makes sense to 
ask the following natural restriction instead. 

\begin{question}\label{Q tf}
Given a surface $\mathbb{S}$, can you find the list of all 
triangle-free graphs having a diameter $2$ 
and domination number at least $3$ 
that can be embedded on $\mathbb{S}$?
\end{question}

Notice that, Plesn\'ik's characterization implies that the
answer for Question~\ref{Q tf} is the empty list when $\mathbb{S}$
is the sphere. On the other hand, the following immediate corollary of Theorem~\ref{th t-free ppg} 
answers the question when $\mathbb{S}$ is the projective plane, along with implying that the domination number of 
triangle-free projective-planar graphs having diameter $2$ is at most three (following our earlier discussions on Euler's characteristic). Note that the domination number of graphs shown in Figure~\ref{fig:family-D} is three.

\begin{theorem}\label{th dom}
Let $G$ be a triangle-free projective-planar graph having a diameter $2$. Then 
\begin{enumerate}
\item[(a)] The domination number $\gamma(G)$ of $G$ is at most $3$. 
    
\item[(b)] If $\gamma(G)=3$, then $G$ is isomorphic to one of the seven graphs depicted in Figure~\ref{fig:family-D}. 
\end{enumerate}
\end{theorem}

As Theorem~\ref{th dom} follows directly from 
Theorem~\ref{th t-free ppg}, we will focus on proving Theorem~\ref{th t-free ppg}.
This is done in Section~\ref{sec proof}. In Section~\ref{sec implications}, we provide some direct implications of our results in determining the absolute clique number of the families of triangle-free projective-planar graphs, which is an important parameter in the theory of homomorphisms of 
colored mixed graphs\footnote{The related definitions are deferred to Section~\ref{sec implications}.}. 

\section{Proof of Theorem~\ref{th t-free ppg}}\label{sec proof}
It is known, due to Euler's formula~\cite{mohar2001} for projective-planar graphs, that any triangle-free 
projective-planar graph 
$G$ has minimum degree $\delta(G) \leq 3$. Therefore, any triangle-free projective-planar graph having a diameter $2$ has a domination number at most $3$.

Notice that as the family of projective-planar graphs is minor-closed, due to The Graph Minor Theorem~\cite{seymour2004}, there exists a finite set $\mathcal{S}$ of graphs such that 
a graph $G$ is projective-planar if and only if $G$ does not contain a minor from $\mathcal{S}$. 
Actually, an explicit description of the set $\mathcal{S}$ is provided in~\cite{archdeacon1981} (see \cite{mohar2001}) and it contains $35$ graphs. 
However, we will not need the full list for our proof - to be precise, we will use only three graphs from that list: (1) $K_{3,5}$, (2) $K_{4,4}^-$, i.e., the graph obtained from $K_{4,4}$ by deleting exactly one edge, and (3) the graph $F_0$ depicted in Figure~\ref{fig:family-F}.  

\begin{figure}
	\centering
	\begin{tabular}{cccc} 
\begin{tikzpicture}[scale=0.45]
	\def\A{0,2}
    \def\B{2,0}
    \def\C{0,-2}
    \def\D{-2,0}
    \def\E{0,1}
    \def\F{1,0}
    \def\G{0,-1}
    \def\H{-1,0}
    \def\I{0,0}
 
    \renewcommand{\vertexset}{(a,\A,black),(b,\B,black),(c,\C,black),(d,\D,black),(e,\E,black),(f,\F,black),(g,\G,black),(h,\H,black),(i,\I,black)};
   \renewcommand{\edgeset}{(a,b),(b,c), (c,d),(a,f),(f,c),(a,h),(h,c),(d,e),(e,b),(d,g),(g,b),(h,i),(i,f),(e,i),(e,g)}
    \renewcommand{\defradius}{.1}
	\drawgraph

\end{tikzpicture}&
\begin{tikzpicture}[scale=0.45]
	\def\A{0,2}
    \def\B{2,0}
    \def\C{0,-2}
    \def\D{-2,0}
    \def\E{0,1}
    \def\F{1,0}
    \def\G{0,-1}
    \def\H{-1,0}
    \def\I{0,0}
 
    \renewcommand{\vertexset}{(a,\A,black),(b,\B,black),(c,\C,black),(d,\D,black),(e,\E,black),(f,\F,black),(g,\G,black),(h,\H,black),(i,\I,black)};
   \renewcommand{\edgeset}{(a,b),(b,c), (c,d),(d,a),(a,f),(f,c),(a,h),(h,c),(d,e),(e,b),(d,g),(g,b),(h,i),(i,f),(e,i),(e,g)}
    \renewcommand{\defradius}{.1}
	\drawgraph

\end{tikzpicture}&
\begin{tikzpicture}[scale=0.45]
	\def\A{0,2}
    \def\B{2,0}
    \def\C{0,-2}
    \def\D{-2,0}
    \def\E{0,1}
    \def\F{1,0}
    \def\G{0,-1}
    \def\H{-1,0}
    \def\I{0,0}
    \def\J{1,-1}
 
    \renewcommand{\vertexset}{(a,\A,black),(b,\B,black),(c,\C,black),(d,\D,black),(e,\E,black),(f,\F,black),(g,\G,black),(h,\H,black),(i,\I,black),(j,\J,black)};
   \renewcommand{\edgeset}{(a,b),(b,j),(j,c), (c,d),(d,a),(a,f),(f,c),(a,h),(h,c),(d,e),(e,b),(d,g),(g,j),(h,i),(i,f),(e,i),(e,g)}
    \renewcommand{\defradius}{.1}
	\drawgraph

\end{tikzpicture}&
\begin{tikzpicture}[scale=0.45]
	\def\A{0,2}
    \def\B{2,0}
    \def\C{0,-2}
    \def\D{-2,0}
    \def\E{0,1}
    \def\F{1,0}
    \def\G{0,-1}
    \def\H{-1,0}
    \def\I{0,0}
    \def\J{1,-1}
    \def\K{-1,-1}
 
    \renewcommand{\vertexset}{(a,\A,black),(b,\B,black),(c,\C,black),(d,\D,black),(e,\E,black),(f,\F,black),(g,\G,black),(h,\H,black),(i,\I,black),(j,\J,black),(k,\K,black)};
   \renewcommand{\edgeset}{(a,b),(b,j),(j,c), (c,k),(d,k),(d,a),(a,f),(f,c),(a,h),(h,c),(k,e),(e,b),(d,g),(g,j),(h,i),(i,f),(e,i),(e,g)}
    \renewcommand{\defradius}{.1}
	\drawgraph

\end{tikzpicture}\\
(a) & (b) & (c) & (d) 
    \end{tabular}
    \caption{The graphs (a) $F_0$, (b) $F_1$, (c) $F_2$, and (d) $F_3$ form the graph family $\mathcal{D}'$.}\label{fig:family-F}
\end{figure}
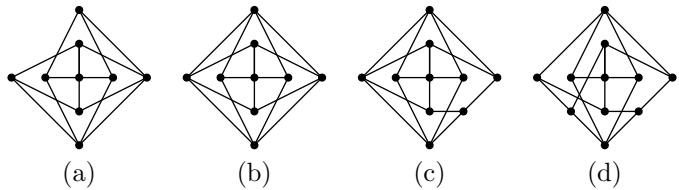

\begin{observation}[\cite{archdeacon1981}; see \cite{mohar2001}]
The graphs $K_{3,5}$, $K_{4,4}^-$, and $F_0$ are not projective-planar. 
\end{observation}

Thus any graph containing $K_{3,5}$, $K_{4,4}^-$, or $F_0$
as a minor is not projective-planar as well. Even though the previous statement is obvious, we will present it as another observation as it will be frequently used in our proofs. 

\begin{observation}[\cite{archdeacon1981}; see \cite{mohar2001}]\label{obs key}
If $G$ contains  $K_{3,5}$, $K_{4,4}^-$, or $F_0$ as a minor, then $G$ is not a projective-planar graph. 
\end{observation}

Now we get into the more technical part of our proof. 
First of all, for convenience, let us denote the family of all 
triangle-free projective-planar graphs having diameter $2$ by 
$\mathcal{PP}_2$. Therefore, what we are trying to do here is to 
provide a list of all graphs in $\mathcal{PP}_2$. We already know that if $G \in \mathcal{PP}_2$, then its
minimum degree $\delta(G)$ is at most $3$. We will use this as the basis of our case analysis. Observe that any $G \in \mathcal{PP}_2$ is connected. Therefore, the logical first step is to handle the graphs having a degree one vertex.

\subsection{Characterizing graphs in $\mathcal{PP}_2$ having minimum degree at most $2$}
\begin{lemma}\label{lm:min-deg-1}
If $\delta(G) = 1$ for a graph $G \in \mathcal{PP}_2$, 
then $G$ is isomorphic to $K_{1,n}$ for some $n \geq 2$. 
\end{lemma}

\begin{proof}
Let $v$ be a degree-$1$ vertex in $G$ having $v_1$ as its only neighbor. As $G$ has diameter $2$, $v_1$ must be adjacent to all the vertices in $V(G) \setminus \{v, v_1\}$. Moreover, as $G$ is triangle-free, the set of all neighbors $N(v_1)$ of $v_1$ 
is an independent set. 
\end{proof}

The next natural step is to consider the graphs having minimum degree equal to $2$.

\begin{lemma}\label{lm:min-deg-2}
If $\delta(G) = 2$ for a graph $G \in \mathcal{PP}_2$, 
then $G$ is isomorphic to $K_{2,n+2}$, $C_5(m,n)$, $K_{3,4}(n)$, or 
$K_{3,3}(n)$  for some $m,n \geq 0$. 
\end{lemma}

\begin{proof}
Let $v$ be a degree-$2$ vertex having  $N(v)=\{v_1,v_2\}$. Let 
$$C = (N(v_1) \cap N(v_2)) \setminus \{v\} \text{ and }
S_i = N(v_i) \setminus (C \cup \{v\})$$
for $i \in \{1,2\}$. 



%
%
%
%


$S_1\cup S_2$ induces a complete bipartite graph (else the end vertices of any non-edge would be at a distance greater than $2$, a contradiction) and as $\delta(G)\geq 2$, both sets are nonempty. 
 
If $|S_1| \geq |S_2| \geq 3$, then we  find a $K^{-}_{4,4}$ by taking the graph induced by $S_1 \cup S_2 \cup \{v_1, v_2\}$. This is a contradiction due to 
Observation~\ref{obs key}. Thus we must have $|S_2| \leq 2$. 

If $|S_2| = 2$, then $|S_1| \leq 3$ as otherwise we can contract the edge $vv_1$ to find a $K_{3,5}$ induced by $S_1 \cup S_2 \cup \{v_1, v_2\}$. This is a contradiction due to 
Observation~\ref{obs key}.

Now observe that $|S_1| = 3$ and $|S_2| = 2$ implies $G$ is isomorphic to $K_{3,4}(n)$, where $n = |C|+1$. Similarly, 
$|S_1| = 2$ and $|S_2| = 2$ implies $G$ is isomorphic to $K_{3,3}(n)$,  where $n = |C|+1$.

If $|S_2|=1$, $|S_1|$ can have any value greater than or equal to $1$. In this case, $G$ is isomorphic to $C_5(m,n)$, where 
$m = |C|$ and $n= |S_1|-1$. 
\end{proof}

\subsection{Characterizing graphs in $\mathcal{PP}_2$ that are $3$-regular}
This leaves us with the final case: considering 
 the graphs having minimum degree equal to $3$. We break this case into two parts, namely, when $G$ is $3$-regular and when $G$ is not $3$-regular, and tackle them separately. Also we will use some new terminologies. 
 
 A vertex $u$ \textit{reaches} a vertex $v$ if they are adjacent or they have a common neighbor. In particular, if $w$ is a common neighbor of $u$ and $v$, then we use the term $u$ reaches $v$  \textit{via} $w$. 

\begin{lemma}\label{lm:3-regular}
If a $3$-regular graph $G \in \mathcal{PP}_2$, 
then $G$ is isomorphic to either $K_{3,3}$ or  $W_8$ or $P_{10}$.
\end{lemma}

\begin{proof}
Let $v \in V(G)$ be any vertex having neighbors 
$\{v_1, v_2, v_3\}$. Moreover, let $S_i$ denote the set of vertices in $G \setminus \{v\}$ which are adjacent to exactly $i$ vertices among $\{v_1, v_2, v_3\}$. 
Note that, as $G$ has diameter $2$, every vertex in 
$G \setminus \{v, v_1, v_2, v_3\}$ belongs to exactly one of $S_1, S_2,$ and $S_3$. 

\medskip 

Observe that as $G$ is $3$-regular, we must have $|S_3| \leq 2$. Moreover, if $|S_3| = 2$, then $G$ is isomorphic to $K_{3,3}$.

\medskip

If $|S_3| = 1$, then note that we must have $|S_2| \leq 1$ as otherwise, it will force one of the vertices among $\{v_1, v_2, v_3\}$ to have at least two neighbors in $S_2$, and hence have at least four neighbors in $G$, contradicting the $3$-regularity of $G$.

Furthermore,  if $|S_3| = |S_2| = 1$, then without loss of generality we may assume that $S_2 = \{u\} \subseteq N(v_1) \cap N(v_2)$. Thus $u$ must reach $v_3$ via some vertex $u' \in S_1 \cap N(v_3)$. 
Now each vertex among $\{v_1, v_2, v_3\}$ already has three neighbors and thus there cannot be any other vertex in $G$. Also all vertices except $u'$ have degree $3$ at present. Thus the $3$-regularity of $G$ forces us to include a new vertex adjacent to $u'$ in the graph, a contradiction. Therefore, $G$ cannot have $|S_3|=|S_2|=1$. 

Thus if $|S_3| = 1$, then we must have $S_2 = \emptyset$. However, due to the $3$-regularity of $G$, each vertex among $\{v_1, v_2, v_3\}$ has exactly one neighbor in $S_1$. Let us assume that the neighbors of $v_1, v_2,$ and $v_3$
in $S_1$ are $w_1, w_2,$ and $ w_3$, respectively. Now, as 
each vertex among $\{v_1, v_2, v_3\}$ already has three neighbors, there cannot be any other vertex in $G$. 
In fact, each vertex of $G$ other than $w_1, w_2,$ and $ w_3$ has degree $3$ already. Thus $w_1, w_2,$ and $ w_3$ must reach all of $v_1, v_2,$ and $ v_3$ either directly or via themselves. That forces $w_1, w_2,$ and $ w_3$ to create a triangle, contradicting the triangle-free property of $G$. Therefore, $G$ cannot have 
$|S_3| = 1$. 

\medskip

Now let us consider the situation where $|S_3| = 0$. 
First observe that $|S_2| \leq 3$ in this case, as otherwise one vertex among $\{v_1,v_2, v_3\}$ has degree at least $4$.

If $|S_2| = 3$, then without loss of generality, we may assume that $S_2 = \{u_{1}, u_{2}, u_{3}\}$ where 
$u_i \in S_2 \setminus N(v_i)$ for all $i \in \{1,2,3\}$. 
Now, as 
every vertex among $\{v_1, v_2, v_3\}$ already has three neighbors each, there cannot be any other vertex in $G$. 
Hence $S_1 = \emptyset$. Thus, in particular, the vertex $u_1$ must reach $v_1$ via some vertex of $S_2$. That will create a triangle, a contradiction. So $|S_2| \leq 2$.

If $|S_2| = 2$ with $S_2 = \{u_1, u_2\}$, then without loss of generality, assume that $v_3$ is a common neighbor of $u_1$ and $u_2$. 
Moreover, the sum of the degrees of $v_1, v_2,$ and $ v_3$ at the moment is $7$ and therefore we must have exactly two more vertices in $G$. Thus, $|S_1| = 2$ and we may assume that $S_1 = \{w_1, w_2\}$. 
Observe that $u_1$ and $ u_2$ must have 
exactly one neighbor each in $S_1$ as they cannot be adjacent to each other in order to avoid creating a triangle. This implies that there are exactly two edges between  the sets $S_1$ and $S_2$. Thus at least one vertex of $S_1$ does not have degree $3$ unless $w_1$ and $w_2$ are adjacents. Hence $w_1$ and $w_2$ must be adjacent. This implies that $w_1$ and $w_2$ do not have a common neighbor. Therefore, without loss of generality, we may assume that $w_i$ and $ u_i$ are adjacent to $v_i$ for all $i \in \{1,2\}$. Hence the edges $u_1w_2$ and $u_2w_1$ are in $G$. Observe that $G$ is isomorphic to $W_8$ in this case.


We have a total of nine vertices, not possible for a 3-regular graph (as nine is an odd number).

 If $|S_2| = 0$, then each $v_i$ has exactly two neighbors $w_i$
 and $w'_i$ in $S_1$ for all $i \in \{1,2,3\}$. 
 Without loss of generality, $w_1$ reaches $v_2$ and $v_3$ via $w_2$ and $w_3$, respectively. As $w_1$ already has three neighbors, 
 $w'_2$ and $w'_3$ must reach $v_1$ via $w'_1$. 
 Note that, $w_2$ cannot reach $v_3$ via $w_3$ in order to avoid a triangle. Therefore, $w_2$ reaches $v_3$ via $w'_3$. Similarly, $w_3$ reaches $v_2$ via $w'_2$. This implies that $G$ is isomorphic to $P_{10}$. 
\end{proof}

\subsection{Characterizing not regular graphs in $\mathcal{PP}_2$ having minimum degree $3$}
Finally, the case where $\delta(G)=3$ and $\Delta(G) \geq 4$ is handled. The proof of Lemma~\ref{lm:max-deg-4} is lengthy; in order to make the proof easier to follow, we have divided it into several claims and lemmas and presented it in a separate subsection. 

\begin{lemma}\label{lm:max-deg-4}
If $\delta(G)=3$ and $\Delta(G) \geq 4$ for a graph 
$G \in \mathcal{PP}_2$, 
then $G$ is isomorphic to $K_{3,4}$, $K_{3,4}^+$, $W^+_8$, $M_{11}$,
$M_{11}^{-}$,  or $M_{11}^{=}$.
\end{lemma}

We will begin by presenting some basic conventions to be used throughout this section. 

\medskip

\subsubsection{Conventions used in the proof of Lemma~\ref{lm:max-deg-4}}
Let $v\in V(G)$ be a vertex with maximum degree and let 
$N(v)=\{v_1, v_2, v_3, v_4\} \cup X$ where $X$ may or may not be $\emptyset$.
Moreover, let $S_i$ denote the set of vertices in $G \setminus \{v\}$ which are adjacent to exactly $i$ vertices among $\{v_1, v_2, v_3, v_4\}$ (see Figure~\ref{fig:conventions}). Furthermore, let $m_2$ be the cardinality of a maximum matching
in $S_2$.

\subsubsection{Basic structural properties}
The proof of Lemma~\ref{lm:max-deg-4} runs via a series of claims and lemmas. In the case of the claims, we always assume that $\delta(G)=3$ and $d(v)=\Delta(G)\geq 4$ for a graph $G \in \mathcal{PP}_2$.

\begin{figure}[h]
\centering
\includegraphics[width=0.2\textwidth]{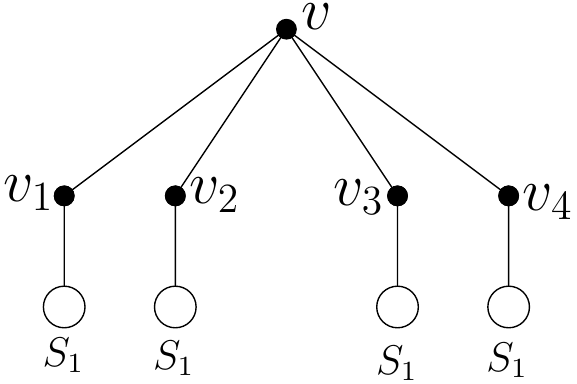}\quad
\includegraphics[width=0.2\textwidth]{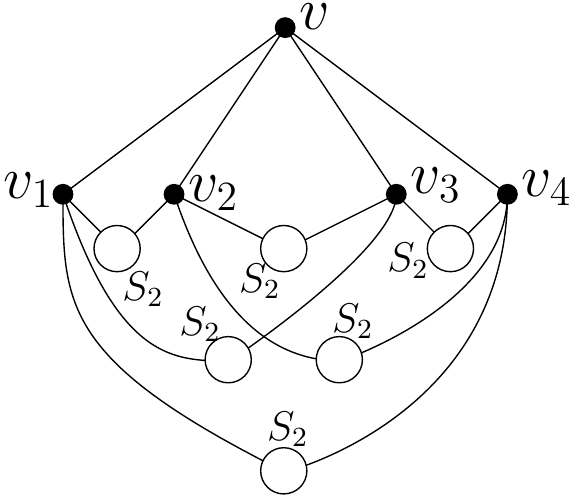}\quad
\includegraphics[width=0.2\textwidth]{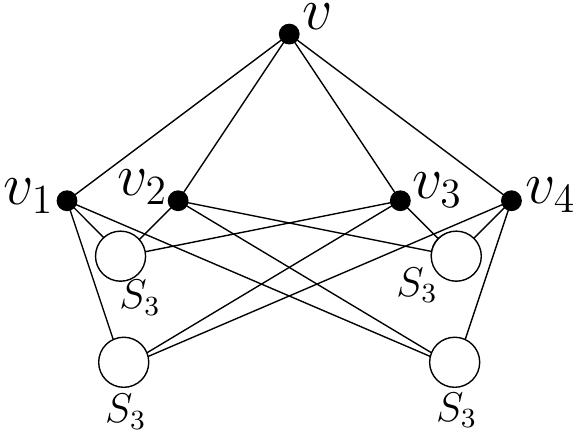}\quad
\includegraphics[width=0.2\textwidth]{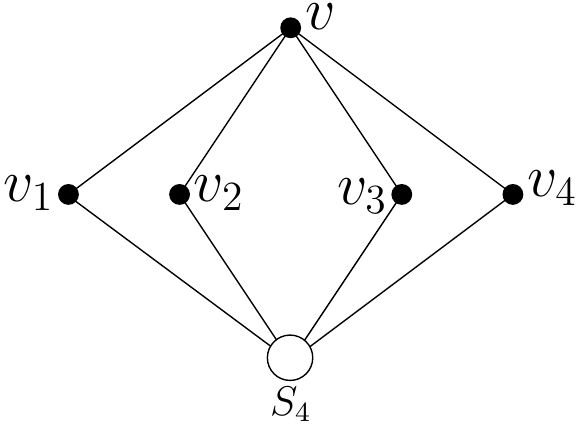}
\caption{Illustrations of our convention}
\label{fig:conventions}
\end{figure}

\medskip

This brings us to our first two observations. 

\begin{claim}\label{claim S4uS3 indep}
	A vertex in $S_3\cup S_4$ is not adjacent to a vertex in $S_2\cup S_3\cup S_4.$
\end{claim}

\begin{proof}
Any vertex in $S_4 \cup S_3$ has at least one neighbor in common with any vertex in $S_2\cup S_3\cup S_4$. Hence, any edge between $S_4 \cup S_3$ and $S_2\cup S_3\cup S_4$ will create a triangle. 
\end{proof}

\medskip

\begin{claim}\label{claim S4+S3+M2 leq 2}
	The value of $|S_4|+|S_3|+m_2$ is at most $2$.
\end{claim}

\begin{proof}
Observe that each vertex of $S_3$ reaches three of the four vertices among $\{v_1,v_2,v_3,v_4\}$
directly and one of them via a vertex of $S_1$ (by Claim~\ref{claim S4uS3 indep}). Therefore, if we contract all the edges between $\{v_1,v_2,v_3,v_4\}$ and $S_1$, then each vertex of $S_3$ becomes adjacent to every vertex of $\{v_1,v_2,v_3,v_4\}$. Moreover, suppose there is an edge having both its end vertices in $S_2$, then these end vertices cannot have a common neighbor, else a triangle is induced. Contracting this edge, the new resulting vertex will correspond to a vertex adjacent to each of $\{v_1,v_2,v_3,v_4\}$. Therefore, if   
$|S_4|+|S_3|+m_2 \geq 3$, then $G$ will contain a $K_{4,4}$-minor. 
\end{proof}

\medskip

Therefore, in particular, the above claim implies that $|S_4|+|S_3| \leq 2$. This bound will be the basis of our case study.

\subsubsection{\textit{Case: }$|S_4|+|S_3| = 2$}

\smallskip
\begin{claim}\label{claim S=2 S2=0}
	If  $|S_4|+|S_3| = 2$, then $S_2 = \emptyset$.
\end{claim}

\begin{proof}
Recall that each vertex of $S_3$ reaches three of the four vertices among $\{v_1,v_2,v_3,v_4\}$
directly and one of them via a vertex of $S_1$ (by Claim~\ref{claim S4uS3 indep}). 
If  $|S_4|+|S_3| = 2$, then $m_2=0$ (by Claim~\ref{claim S4+S3+M2 leq 2}). Now if $S_2 \neq \emptyset$, then each vertex of $S_2$ reaches two of the four vertices among $\{v_1,v_2,v_3,v_4\}$
directly and two of them via vertices of $S_1$. This is implied by Claim~\ref{claim S4uS3 indep} and the following: any vertex in $S_2$ cannot reach the two non-adjacent vertices in  $\{v_1,v_2,v_3,v_4\}$ via a vertex in $S_2$ (as $m_2=0$). 
Hence, contract all edges between $S_1$ and 
$\{v_1, v_2, v_3, v_4\}$ to obtain a $K_{4,4}$, a contradiction. 
\end{proof}
 
\medskip

\begin{claim}\label{claim S=2 S1 restricted}
	If  $S_4 \cup S_3 = \{u_1, u_2\}$ and $v_i$ is a common neighbor of $u_1$ and $u_2$ for some $i \in \{1,2,3,4\}$, then no vertex  $w \in S_1$ is adjacent to $v_i$. 
\end{claim}

\begin{proof}
By Claim~\ref{claim S=2 S2=0}, we know that $S_2=\emptyset$. Therefore, each vertex in $S_4\cup S_3$ would reach non-adjacent vertices in $\{v_1,v_2,v_3,v_4\}$ via some vertices in $S_1$. 
If there exists a vertex $w \in S_1 \cap N(v_i)$, where $v_i$ is as defined in the lemma statement, then $w$ cannot be adjacent to any of $\{u_1, u_2\}$ (else a triangle is induced). Since $S_2=\emptyset$ (by Claim~\ref{claim S=2 S2=0}), the vertex $w$ reaches vertices in $\{v_1, v_2, v_3, v_4\}\setminus \{v_i\}$ via vertices in $S_1$. Contract all the edges between $S_1$ and $\{v_1, v_2, v_3, v_4\}$ except $wv_i$ to obtain a $K_{4,4}$, a contradiction. 
\end{proof}

\medskip

\begin{claim}\label{claim S=2 X=0}
	If  $|S_4|+|S_3| = 2$, then $X = \emptyset$.
\end{claim}
\begin{proof}
Assume that $S_4 \cup S_3 = \{u_1, u_2\}$ and there exists an $x \in X$. 

If $|S_4| = 2$, then $x$ reaches $u_1$ and $u_2$ directly or via some other vertices not adjacent to $\{v_1, v_2, v_3, v_4\}$ (else a triangle is induced). Contract the edges between $x$ and those vertices (if they exist) via which $x$ reaches $u_1, u_2$ to obtain a $K_{3,5}$, a contradiction.

If $|S_4| = 1$ and without loss of generality 
$u_1 \in S_4$ and $u_2 \in S_3$, then $x$ reaches $u_1$ directly or via some other vertex $w_1$ not adjacent to $\{v_1, v_2, v_3, v_4\}$ (else a triangle is induced). Contract the edge  $xw_1$ 
if it exists. Observe that $u_2$ is adjacent to some vertex $w_2 \in S_1$ in order to reach all of $\{v_1, v_2, v_3, v_4\}$. Contract $u_2w_2$. Note that $x$ reaches $u_2$ directly, or via $w_2$ or via some other vertex $w_3$ which is not adjacent to $\{v_1, v_2, v_3, v_4\}$. 
Contract $xw_3$, if it exists, to obtain  a $K_{3,5}$, a contradiction.

If $|S_4| = 0$ and thus $u_1 , u_2 \in S_3$, then 
$u_1$ and $u_2$  may be non-adjacent to the same or different vertices in $\{v_1,v_2,v_3,v_4\}$. 

\begin{itemize}
    \item[Case 1.] If they are non-adjacent to different vertices of $\{v_1,v_2,v_3,v_4\}$, then without loss of generality assume that $u_i$ is not adjacent to $v_i$ for $i \in \{1,2\}$. In this case, $u_i$
must reach $v_i$ via some $w_i \in S_1$ (by Claim~\ref{claim S4uS3 indep}) and $x$ must reach $u_i$ directly or via $w_i$ or via some $w'_i\in S_1$ with $w'_i \not\in \{w_1, w_2\}$. Contract the edges $u_iw_i$ and $xw'_i$ (if they exist) for all $i \in \{1,2\}$ to obtain  a $K_{3,5}$, a contradiction.

\item[Case 2.]  If they are non-adjacent to the same vertex, without loss of generality say, $v_4$ of $\{v_1,v_2,v_3,v_4\}$, then $v_4$ must have at least two neighbors $w_1, w_2$ as the degree of $v_4$ is at least three. Also $w_1, w_2\in S_1$ (by Claim~\ref{claim S4uS3 indep}). Due to Claims~\ref{claim S=2 S2=0} and~\ref{claim S=2 S1 restricted}, we know that the only way for $w_i$ to reach $v_1$ is via $u_1$ or $u_2$. Moreover, for each $i \in \{1,2\}$, 
$u_i$ must reach $v_4$ via some vertex of $S_1$ (by Claim~\ref{claim S4uS3 indep}). Therefore, we must have a perfect matching between 
$\{u_1, u_2\}$ and $\{w_1, w_2\}$. Without loss of generality, assume that perfect matching be 
$\{u_1w_1, u_2w_2\}$. Furthermore, observe that if $x$ is adjacent to any one of $u_1$ or $u_2$, we can rename the vertices of $N(v)$ to reduce this to a case where $|S_4| \geq 1$, which we have already handled earlier in the proof of this lemma. Therefore, $x$ must reach $u_i$ via $w_i$ or via some vertex $w'_i \not\in \{u_1, u_2, w_1, w_2\}$ for all $i \in \{1,2\}$.  Contract the edges $u_iw_i$ and $xw'_i$ (if they exist) for all $i \in \{1,2\}$ to obtain  a $K_{3,5}$, a contradiction.   
\end{itemize}

This concludes the proof.
\end{proof}

\medskip

Now that we have shown the set $X = \emptyset$ when $|S_4|+|S_3| = 2$, we can try to characterize the graphs for this case.

\medskip

\begin{lemma}\label{claim S4=2 implies K34}
	If $\delta(G)=3$ and $d(v)=\Delta(G)\geq 4$ for a graph $G \in \mathcal{PP}_2$, then the following holds: if  $|S_4| = 2$, then $G$ is isomorphic to $K_{3,4}$.
\end{lemma}

\begin{proof}
If  $|S_4| = 2$, then $X = S_3 = S_2 = S_1 = \emptyset$
due to Claims~\ref{claim S4+S3+M2 leq 2}, \ref{claim S=2 S2=0}, \ref{claim S=2 S1 restricted} and~\ref{claim S=2 X=0}. The only vertices in the graphs other 
than $\{v, v_1, v_2, v_3, v_4\}$ are the 
two vertices in $S_4$. Thus $G$ is isomorphic to $K_{3,4}$.
\end{proof}

\begin{claim}\label{claim S4=S3=1 implies nothing}
	It is not possible to have $|S_4| = |S_3| = 1$. 
\end{claim}

\begin{proof}
If  $|S_4| = |S_3| = 1$, then without loss of generality assume that $u_1 \in S_4$, $u_2$ $\in S_3\setminus N(v_4)$. 
Observe that $X = S_3 = S_2 = \emptyset$
due to Claims~\ref{claim S4+S3+M2 leq 2}, \ref{claim S=2 S2=0} and~\ref{claim S=2 X=0}. 
Moreover, all the  vertices in $S_1$ are adjacent to $v_4$ due to Claim~\ref{claim S=2 S1 restricted}. As $u_2$ must reach $v_4$ via some vertex of $S_1$, the set $S_1$ is not empty. However, any vertex of $S_1$ has exactly two neighbors, namely $v_4$ and $u_2$, contradicting $\delta(G) = 3$.
\end{proof}

\begin{lemma}\label{claim S3=2 implies W8+}
	If $\delta(G)=3$ and $d(v)=\Delta(G)\geq 4$ for a graph $G \in \mathcal{PP}_2$, then the following holds: if $|S_3| = 2$, then $G$ is isomorphic to $W^+_8$. 
\end{lemma}

\begin{proof}
Assume that $S_3 = \{u_1, u_2\}$. Thus 
$X = S_4 = S_2 =  \emptyset$
due to Claims~\ref{claim S=2 X=0}, \ref{claim S4+S3+M2 leq 2}
and~\ref{claim S=2 S2=0}.  Observe that it is enough to consider the following two cases: (i) $u_1$ and $u_2$ are both non-adjacent to $v_4$, and (ii) $u_i$ is non-adjacent to $v_i$ for $i \in \{1,2\}$. 

\begin{itemize}
    \item[Case (i).] If $u_1$ and $u_2$ are both non-adjacent to $v_4$, 
    then all the  vertices in $S_1$ are adjacent to $v_4$ due to Claim~\ref{claim S=2 S1 restricted}. The vertices $v, v_1, v_2, v_3, v_4, u_1,$ and $u_2$ and the vertices belonging to $S_1$ are all of the vertices of $G$. As $\delta(G) = 3$, each vertex of $S_1$ must be adjacent to both $u_1$ and $u_2$. Furthermore, $S_1$ must have at least two vertices, say $w_1$ and $w_2$, as $\delta(G) = 3$. 
    Now contract $vv_4$ to obtain  a $K_{3,5}$ induced by 
    $\{u_1, u_2, (vv_4)\} \sqcup \{v_1, v_2, v_3, w_1, w_2\}$ where $(vv_4)$ denotes the new vertex obtained by contracting the edge $vv_4$. Thus this case is not possible since $K_{3,5}$ is not projective planar.  
    
    \item[Case (ii).] If $u_i$ is non-adjacent to $v_i$ for $i \in \{1,2\}$, 
    then all the  vertices in $S_1$ are adjacent to either $v_1$ or $v_2$ due to Claim~\ref{claim S=2 S1 restricted}. Therefore, 
    the vertices $v, v_1, v_2, v_3, v_4, u_1,$ and $u_2$ and the vertices belonging to $S_1$ are all of the vertices of $G$. 
    Suppose that 
    $S_1=\{w_1, w_2, \ldots, w_k, w'_1, w'_2, \ldots, w'_r\}$ where $w_i$s are adjacent to $v_1$ and $w'_j$s are adjacent to $v_2$ where 
    $(i,j) \in \{1,2, \ldots, k\} \times \{1,2, \ldots, r\}$. 
    Each $w_i$ reaches $v_3$ via $u_1$ and each $w'_j$ is reaches $v_3$ via $u_2$ for all 
    $(i,j) \in \{1,2, \ldots, k\} \times \{1,2, \ldots, r\}$.
    Furthermore as $\delta(G) = 3$, each $w_i$ must be adjacent to some $w'_j$ and    each $w'_p$ must be adjacent to some $w_q$. 
    Without loss of generality and due to symmetry, we may assume that 
    $1 \leq r \leq k$.  
        If $k \geq 2$, then 
    contract the edges $u_2w'_j$ for all $j \in \{1,2,\ldots, r\}$ 
    to obtain the new vertex $(u_2w'_j)$, 
    and contract the edge $vv_1$ to obtain the new vertex $(vv_1)$. 
     Observe that,  in this contracted graph, the  vertices 
     $\{u_1, (vv_1), (u_2w'_j)\} \sqcup \{w_1, w_2, v_2, v_3, v_4\}$ induce a $K_{3,5}$ subgraph, a contradiction. 
    Therefore, $k = r = 1$. Thus  $G$ is isomorphic to $W_8^+$. 
\end{itemize}

This ends the proof of the lemma.
\end{proof}

This concludes the case when we have $|S_4|+|S_3|=2$. We will present the summary of it in the following lemma.

\begin{lemma}\label{lem case s3+s4=2}
	If $\delta(G)=3$ and $d(v)=\Delta(G)\geq 4$ for a graph $G \in \mathcal{PP}_2$, then the following holds: if ${|S_4|+|S_3|=2}$, then $G$ is isomorphic to $K_{3,4}$ or $W_8^{+}$. 
\end{lemma}

\begin{proof}
Follows directly from Lemma~\ref{claim S4=2 implies K34}, Claim~\ref{claim S4=S3=1 implies nothing}, and Lemma~\ref{claim S3=2 implies W8+}. 
\end{proof}

\medskip

It remains to analyse the situations when $|S_4|+|S_3| \leq 1$. 
The first case is when $|S_4|+|S_3| = 1$.

\medskip
\subsubsection{\textit{Case:} $|S_4|+|S_3| = 1$}
\smallskip

\begin{claim}\label{claim S4+S3=1 implies S2neq0}
	If $|S_4|+|S_3| = 1$, then $S_2 \neq \emptyset$.
\end{claim}

\begin{proof}
Let us assume the contrary and suppose that 
$S_2 = \emptyset$. 
Furthermore, assume that $S_4 \cup S_3 = \{u_1\}$ and that $u_1$ is adjacent to $\{v_1, v_2, v_3\}$. 
First we will show that it is not possible for 
any $v_i$, for $i \in \{1,2,3\}$, to have 
two neighbors in $S_1$.

Hence without loss of generality assume that  $v_1$ is 
adjacent to  $w_1, w_2  \in S_1$. Observe that 
both $w_1$ and $w_2$ must reach $v_2, v_3,$ and $v_4$ via some vertices from $S_1 \setminus \{w_1, w_2\}$ and $u_1$ must reach $v_4$ via some vertex from 
$S_1 \setminus \{w_1, w_2\}$ (all due to Claim~\ref{claim S4uS3 indep}). Now contract all the edges between $\{v_2, v_3, v_4\}$ and 
$S_1 \setminus \{w_1, w_2\}$ to obtain a $K_{4,4}$-minor, a contradiction. 
Thus each $v_i$s, for $i \in \{1,2,3\}$, 
can have at most one neighbor in $S_1$. 

However,  $\delta(G) = 3$ implies that each $v_i$ must be adjacent to exactly one vertex (say) $w_i$ from $S_1$ for $i \in \{1,2,3\}$. Now note that $w_1$ must reach $v_2$ and $v_3$ via $w_2$ and $w_3$, respectively and $w_2$ must reach $v_3$ via $w_3$. 
This creates a triangle induced by $\{w_1, w_2, w_3\}$ in $G$, a contradiction. 
\end{proof}

\medskip

\begin{claim}\label{claim S4=1 implies X=0}
	If $|S_4| = 1$ and $|S_3| = 0$, then $X = \emptyset$.
\end{claim}

\begin{proof}
Let $S_4 = \{u_1\}$ and let $x \in X$. Claim~\ref{claim S4+S3=1 implies S2neq0} implies the existence of a vertex $w_1 \in S_2$. Without loss of generality
assume that $w_1$ is adjacent to $v_1$ and $v_2$. 

Note that $x$ reaches $u_1$ directly or via some vertex $a$ (say) not adjacent to any of $\{v_1, v_2, v_3, v_4\}$ 
in order to avoid creating a triangle in $G$. 
Moreover, $w_1$ reaches $x, v_3,$ and $v_4$ via some vertices not adjacent to any of $\{v_1,v_2\}$. Let $A$ denote the set of vertices via which 
$w_1$ reaches $x, v_3,$ and $v_4$. Contract the edges between $A\setminus \{a\}$ and $w_1$ and the edge $xa$. The vertices $\{v,w_1,u_1\}\sqcup \{x, v_1,v_2,v_3,v_4\}$ form the partition of a $K_{3,5}$-minor, a contradiction.
\end{proof}

\medskip

\begin{claim}\label{claim S4=1 implies m2=1}
	If $|S_4| = 1$ and $|S_3| = 0$, then $m_2 = 1$.
\end{claim}

\begin{proof}
By Claim~\ref{claim S4+S3+M2 leq 2}, $m_2 \leq 1$.
Next, suppose that $S_4 = \{u_1\}$ and $m_2 = 0$. 
That means every vertex of $S_2$ (which is non-empty by Claim~\ref{claim S4+S3=1 implies S2neq0}) must reach its non-adjacent $v_i$s via vertices of $S_1$ (by Claim~\ref{claim S4uS3 indep} and $m_2=0$). Therefore, if $|S_2| \geq 2$, then contracting the edges between $S_1$ and $\{v_1, v_2, v_3, v_4\}$ will create a $K_{4,4}$-minor. 

Thus, we have $|S_2|=1$. Without loss of generality assume that $S_2 = \{w\}$ and that $w$ is adjacent to $v_1$ and $v_2$. 

Notice that $w$ must reach $v_3$ and $v_4$ via $w_3, w_4 \in S_1$, respectively (by Claim~\ref{claim S4uS3 indep} and $m_2=0$). Thus to avoid creating a triangle, $w_3$ must reach $v_4$ via $w_4' \in S_1$. Moreover, to avoid creating a triangle, $w_4'$ must reach $v_1$ via some $w_1 \in S_1$. 
    
Notice that, $w_1$ reaches $v_2, v_3,$ and $v_4$ via some elements of $S_1$. Thus, if we contract all the edges between $S_1$ and $\{v_2, v_3, v_4\}$, we will create a $K_{4,4}$-minor.

Therefore, $m_2 \geq 1$. Since $m_2 \leq 1$ by Claim~\ref{claim S4+S3+M2 leq 2}, we have $m_2 = 1$. 
\end{proof}

\medskip

\begin{lemma}\label{claim S4=1 S2>0}
	If $\delta(G)=3$ and $\Delta(G)\geq 4$ for a graph $G \in \mathcal{PP}_2$, then the following holds: if $|S_4| = 1$ and $|S_3| = 0$, then $G$ is isomorphic to $K_{3,4}^{*}$. 
\end{lemma}

\begin{proof}
Let $S_4 = \{u_1\}$ and, thus, by Claim~\ref{claim S4=1 implies m2=1}
we know that $m_2 = 1$. 
Then without loss of generality we may assume the existence of an edge $w_1w_2$ such that $w_1, w_2 \in S_2$, 
$w_1$ is adjacent to $v_1$ and $v_2$, and 
$w_2$ is adjacent to $v_3$ and $v_4$. 
If there are no other vertex or edge in $G$, 
then $G$ is isomorphic to $K_{3,4}^{*}$. 

However, if there is another vertex $w_3 \in S_2$ and 
if $w_3$ reaches $\{v_1, v_2, v_3, v_4\}$ directly or 
via some vertices except $w_1$ and $w_2$, then contract the edge 
$w_1w_2$. Also contract the edges between $w_3$ and the vertices via which $w_3$ reaches $v_i$s, for $i \in \{1,2,3,4\}$. This will result in a $K_{4,4}$-minor. 

On the other hand, 
if there is a $w_3 \in S_2$ and 
if $w_3$ reaches $\{v_1, v_2, v_3, v_4\}$ directly or 
via  $w_i$ for some $i \in \{1,2\}$, then $G$ is
the graph $F_1$ (depicted in Fig.~\ref{fig:family-F}(b))
that contains the graph $F_0$ (depicted in Fig.~\ref{fig:family-F}(a)) as a subgraph, and thus as a minor. 
\end{proof}

So far we were dealing with the case when $G$ is a graph with $|S_4|=1$ and $|S_3|=0$. Now we turn our attention towards the case when $G$ is a graph with $|S_4|=0$ and $|S_3|=1$.
Initially, we will observe some properties that this condition implies. However, finally, the satisfaction of those properties will turn out to be impossible, thereby proving that there are no required graphs with $|S_4|=0$ and $|S_3|=1$.

\begin{claim}\label{claim S3=1 X=0}
	If $|S_4| = 0$ and $|S_3| = 1$, then $X= \emptyset$. 
\end{claim}
\begin{proof}
Suppose that $X \neq \emptyset$ and $x \in X$. 
Let $S_3 = \{u_1\}$ and without loss of generality 
let $u_1$ be adjacent to $v_1, v_2, v_3$. Therefore, $u_1$ must 
reach $v_4$ via $w_4 \in S_1$ (by Claim~\ref{claim S4uS3 indep}). 
We know that $S_2 \neq \emptyset$ due to 
Claim~\ref{claim S4+S3=1 implies S2neq0}. Also let $u_2 \in S_2$.

If $u_2$ reaches $v_4$ directly or via any vertex other than $w_4$, 
then contract the edge $u_1w_4$, 
all the edges connecting $u_2$ to its neighbors via which $u_2$ reaches $v_1, v_2, v_3$ or $v_4$, and 
all the edges connecting $x$ to its neighbors via which $x$ reaches $u_1$ and $u_2$,
in order to obtain a $K_{3,5}$-minor (see the vertices in the partition $\{v,u_1,u_2\}\sqcup \{x,v_1,v_2,v_3,v_4\}$).

Thus $u_2$ reaches $v_4$ via $w_4$ and nothing else. 
Hence without loss of generality, we may assume that $u_2$ is adjacent 
to $v_1$ and $v_2$ and that $u_2$ reaches $v_3$ via $w_3 \in S_1$ (if $w_3\in S_2$, then $w_3$ must be adjacent to $v_3,v_4$ and by similar arguments as the previous case existence $K_{3,5}$ minor can be shown). 
Moreover, as $\delta(G) = 3$, there must be a $w'_4 \in S_1$ adjacent to $v_4$. On the other hand, $v_1$ cannot have a neighbor in $S_1$, as otherwise we may contract all the edges between $S_1$ and $\{v_2, v_3, v_4\}$ to obtain a $K_{4,4}$-minor. Therefore, $w'_4$ must reach $v_1$ via $u_1$ (if $w'_4$ reaches $v_1$ via $u_2$, then a $K_{3,5}$-minor is formed: see the vertices in the partition $\{v,u_1,u_2\}\sqcup \{x,v_1,v_2,v_3,v_4\}$). 
Now contract the edges $u_1w'_4, u_2w_3, u_2w_4$, and 
all the edges connecting $x$ to its neighbors via which 
$x$ reaches $u_1, u_2$. If $x$ is adjacent to $w_4$, then also contract the edge connecting $x$ to its neighbor via which $x$ reaches $w'_4$. This creates a $K_{3,5}$-minor (see the vertices in the partition $\{v,u_1,u_2\}\sqcup \{x,v_1,v_2,v_3,v_4\}$).  
\end{proof}

\medskip

\begin{claim}\label{claim S3=1 m2=1}
	If $|S_4| = 0$ and $|S_3| = 1$, then $m_2= 1$. 
\end{claim}

\begin{proof}
We already know that $m_2= 0$ or $1$ due to Claim~\ref{claim S4+S3+M2 leq 2}. 

If $m_2 = 0$, then $|S_2| \leq 1$ as otherwise we can contract all the edges between $S_1$ and $\{v_1, v_2, v_3, v_4\}$ to obtain a $K_{4,4}$-minor. 
Thus $|S_2|=1$ due to Lemma~\ref{claim S4+S3=1 implies S2neq0}. 
Without loss of generality assume that $S_3 = \{u_1\}$, $S_2 = \{u_2\}$, $u_1$ is adjacent to $v_1, v_2, v_3$ and $u_2$ is adjacent to $v_1$. 
If $v_1$ is adjacent to some vertex of $S_1$, then we can contract all the edges between $S_1$ and $\{v_2, v_3, v_4\}$ and obtain a $K_{4,4}$-minor. Thus $v_1$ cannot have a neighbor in $S_1$. Hence every vertex of $S_1$ must reach $v_1$ via $u_1$ or $u_2$.

If $u_2$ is adjacent to $v_4$ as well, then $u_2$ must reach $v_2$, $v_3$ via some $w_2, w_3 \in S_1$, respectively. Now $w_2$ must reach 
$v_3$ via some vertex $w'_3 \in S_1$. Observe that it is not possible to have $w_3 = w'_3$ as $G$ is triangle-free. Now $w'_3$ must reach $v_1$ via $u_1$ or $u_2$. In any case, this will create a triangle. Therefore, $u_2$ is not adjacent to $v_4$. Thus we may assume without loss of generality that $u_2$ is adjacent to $v_2$.

If $u_2$  is adjacent to $v_2$, 
then $u_2$ must reach $v_3$ via some $w_3 \in S_1$
and $w_3$ must reach $v_4$ via some $w_4 \in S_1$. 
Observe that $u_2$ cannot be adjacent to $w_4$ 
in order to avoid creating a triangle. 
Therefore, $u_2$ must reach $v_4$ via some $w'_4 \in S_1$. 
Note that $w'_4$ must reach $v_3$ via some 
distinct $w'_3 \in S_1$
in order to avoid creating triangle. 

By what we have already noted above in this proof, we know that 
the only way for $w'_3$ to reach $v_1$ 
is via $u_1$ or $u_2$. In each case a triangle will be created. 
\end{proof}

\medskip

\begin{claim}\label{claim S3=1 not possible}
If $|S_4| = 0$, then it is not possible to have $|S_3| = 1$.
\end{claim}

\begin{proof}
Suppose the contrary. We already know that $m_2= 1$ due to Claim~\ref{claim S3=1 m2=1}. 
Thus note that without loss of generality we may assume that 
$S_3 = \{u_1\}$, $u_2, u_3 \in S_2$, $u_2u_3 \in E(G)$, 
$u_1$ is adjacent to $\{v_1, v_2, v_3\}$, $u_2$ is adjacent to $\{v_1, v_2\}$,
$u_3$ is adjacent to $\{v_3, v_4\}$ and 
$u_1$ reaches $v_4$ via $w_1 \in S_1$. 

Observe that any  vertex in $S_2 \setminus \{u_2, u_3\}$ will force a $K_{4,4}$-minor or a $F_2$ (depicted in Fig.~\ref{fig:family-F}(c) which contains $F_0$, depicted in Fig.~\ref{fig:family-F}(a), as a minor) as a subgraph of $G$ (this is similar to the second half of the proof of Lemma~\ref{claim S4=1 S2>0}). Thus we may infer that $S_2 = \{u_2, u_3\}$. Note that, according to the partial description of $G$ till now $w_1$ has two neighbors. Due to the minimum degree requirement, it must have another neighbor. If $u_2$ is a neighbor of $w_1$, then note that $v_1, v_2, v_3$, and $w_1$ are neighbors of $u_1$ such that $v$ is adjacent to three of them and $u_2$ is adjacent to three of them. This reduces the case to where $|S_3| = 2$, which is already taken care of. 

Therefore, $w_1$ is adjacent to another vertex $w_2 \in S_1$. Note that, if $w_2$ reaches $v_1, v_2, v_3,$ and $v_4$ directly or via vertices from $S_1$, then contracting the edge $u_2u_3$ and all the edges between $S_1$ and $\{v_1, v_2, v_3, v_4\}$ except for the edge having $w_2$ as an endpoint  creates a $K_{4,4}$-minor. If $w_2$ is adjacent to either of $v_1$ or $v_2$ (without loss of generality assume it is adjacent to $v_1$), then it reaches $v_2$ and $v_3$ via vertices of $A\subseteq S_1$. Now contracting edges $w_1v_4$, $u_2u_3$ and edges between $w_2$ and $A$, we get a $K_{4,4}$-minor. 
Thus, the following situation is forced: $w_2$ is adjacent to $v_3$ and $u_2$. This creates the subgraph $F_2$ (depicted in Fig.~\ref{fig:family-F}(c)) in $G$ which contains $F_0$ 
(depicted in Fig.~\ref{fig:family-F}(a)) as a minor, a contradiction.  
\end{proof}

This concludes the case when we have $|S_4|+|S_3|=1$. We will present the summary of it in the following lemma.

\begin{lemma}\label{lem case s3+s4=1}
	If $\delta(G)=3$ and $d(v)=\Delta(G)\geq 4$ for a graph $G \in \mathcal{PP}_2$, then the following holds: if ${|S_4|+|S_3|=1}$, then $G$ is isomorphic to $K_{3,4}^{*}$ or $W_8^{+}$. 
\end{lemma}

\begin{proof}
Follows directly from Lemma~\ref{claim S4=1 S2>0} and Claim~\ref{claim S3=1 not possible}. 
\end{proof}

\medskip

This brings us to the case where $|S_4|+|S_3| = 0$. 

\medskip

\subsubsection{\textit{Case:} $|S_4|+|S_3| = 0$}

\smallskip

\begin{claim}\label{S4=S3=S2=0 not possible}
	It is not possible to have  $|S_4|=|S_3|=|S_2| = 0$.
\end{claim}

\begin{proof}
As $\delta(G) \geq 3$, each $v_i$ must have at least two neighbors in $S_1$. Thus without loss of generality assume that $v_i$ is adjacent to 
$w_i,w'_i  \in S_1$ for all $i \in \{1,2,3,4\}$. Moreover, without loss of generality, we may suppose that 
$w_1$ reaches $v_i$ via $w_i$ for all $i \in \{2,3,4\}$. Note that 
as $G$ is triangle-free, $\{w_2, w_3, w_4\}$ is an independent set. 
Therefore, contracting the edges between the vertices of 
$\{v_1, v_2, v_3, v_4\}$ 
and the vertices of $(S_1 \setminus \{w_2, w_3, w_4\})$ creates a $K_{4,4}$-minor. 
\end{proof}

Now we will  consider the case when $|S_2| \geq 1$.

\medskip

\begin{claim}\label{S4=S3=0, S2 has 3 with one vi}
	If $|S_4|=|S_3| = 0$, then it is not possible for $v_i$, for all $i\in \{1,2,3,4\}$, to have three or more 
	neighbors in $S_2$. 
\end{claim}

\begin{proof}
Let us assume the contrary. Without loss of generality  suppose 
that $v_1$ is adjacent to $u_1, u_2, u_3 \in S_2$. Furthermore suppose that $u_1$ is adjacent to $v_2$ as well. Observe that 
$\{u_1, u_2, u_3\}$ is an independent set as $G$ is triangle-free. 

If $u_2$ or $u_3$ is also adjacent to $v_2$, 
then by renaming $v_1$ as $v$, the case reduces to $|S_3|+|S_4| \geq 1$ which has been handled before. 

Thus without loss of generality, we may assume that 
$u_2$ is adjacent to $v_3$ and 
$u_3$ is adjacent to $v_4$. 
Then $u_1$ must reach $v_3$ and $v_4$; $u_2$ must reach $v_2$ and $v_4$; and $u_3$ 
must reach $v_2$ and $v_3$, 
via some vertices of $S_1 \cup S_2$. 
If they use vertices from $S_2$, then those vertices must be distinct. Let $A$ be the vertices via which $u_1, u_2,$ and $u_3$
reach $v_2, v_3,$ and $v_4$. Contract the edges between $A \cap S_2$ and $\{u_1,u_2,u_3\}$. Also, contract the edges between $A \cap S_1$
and $\{v_2, v_3, v_4\}$. 
We will obtain a $K_{4,4}$-minor, a contradiction. 

Thus we have considered all the cases up to symmetry and have proved the claim. 
\end{proof}

\medskip 







\begin{claim}\label{S2 have 2 with same parents, not possible}
	If $|S_4|=|S_3| = 0$ , then it is not possible to have $u_1, u_2 \in S_2$ 
	having $N(u_1) \cap \{v_1,v_2,v_3,v_4\} = N(u_2) \cap \{v_1,v_2,v_3,v_4\}$.  
\end{claim}
\begin{proof}
Let us assume the contrary. Without loss of generality  suppose 
that $u_1$ and $u_2$  are adjacent to both $v_1$ and $v_2$.

Note that it is not possible to have any vertex other than $v, u_1$, and $u_2$ adjacent to 
$v_1$ (or  $v_2$) as otherwise our case will get reduced to the case where $|S_3|+|S_4| \geq 1$ by renaming $v_1$ as $v$ which we have already taken care of.

Therefore, every vertex from 
$V(G) \setminus (N[v] \cup \{u_1,u_2\}) = A$ (say)  is adjacent to
either $u_1$ or $u_2$ in order to reach $v_1$ and $v_2$. 
Since $\delta(G)\geq 3$, $v_3$ has two more neighbors. Both these neighbors are adjacent to $u_1$ or $u_2$. 
If any vertex from $A$ is adjacent to both $u_1$ and $u_2$, then one of $u_1$ or $u_2$, without loss of generality assumes $u_1$, has degree 4. Then 
our case will get reduced to the case where $|S_3|+|S_4| \geq 1$ by renaming $u_1$ as $v$ which we have already taken care of.

As diameter of $G$ is $2$, $u_1, u_2$ must reach $N(v) \setminus \{v_1,v_2\}$ via some vertices from $A$. Let $A_i$ be the set of vertices from $A$ via which $u_i$ reaches the vertices of 
$N(v) \setminus \{v_1,v_2\}$ where $i \in \{1,2\}$. 
Due to the observation made in the previous paragraph, we know that the sets $A_1$ and $A_2$ are disjoint.

If $X \neq \emptyset$, then contract the edges between $A_i$ and $\{u_i\}$ for each $i \in \{1,2\}$ to obtain 
a $K_{3,5}$-minor, a contradiction (see the vertices in the partition $\{v,u_1,u_2\}\sqcup \{x,v_1,v_2,v_3,v_4\}$, where $x\in X$). 
Thus we may assume that $X = \emptyset$. 

\medskip

If there exists $u_3 \in S_2$, it must be adjacent to both 
$u_1$ and $u_2$ in order to reach them. This follows from Claim~\ref{S4=S3=0, S2 has 3 with one vi}. But we have already shown that this is not possible. Thus there are no vertices in $S_2$ other than $u_1$ and $u_2$. 

However as $\delta(G) \geq 3$, there are at least two neighbors 
$w_{i1}, w_{i2} \in S_1$ of $v_i$ for $i \in \{3,4\}$. 
Without loss of generality suppose that $w_{31}$ reaches $v_1$ and $v_2$ via $u_1$. 
Therefore, $w_{31}$ have to reach $v_4$ via a vertex of $S_1 \cap N(v_4)$, say $w_{41}$. Now as $G$ is triangle free, $w_{41}$ reaches $v_1$ and $v_2$ via $u_2$. 
If $w_{32}$ is adjacent to $w_{41}$, then a triangle is induced as $w_{32}$ has to be adjacent to $u_1$ or $u_2$ in order to reach $v_1$ and $v_2$. Thus $w_{32}$ is not adjacent to $w_{41}$. 
Next, observe that 
$w_{32}$ must reach $w_{41}$ via $u_2$. 
Finally, $w_{32}$ must reach $v_4$ via some vertex in $S_1 \cap N(v_4)$, say $w_{42}$ and, then  $w_{42}$ must reach $v_1$ and $v_2$ via $u_1$. 

This so-obtained graph is isomorphic to the graph $F_3$ depicted in Fig.~\ref{fig:family-D}. 
\end{proof}

\medskip

\begin{claim}\label{S4=S3=0, S2 has 3 with one vi missing}
	If $|S_4|=|S_3| = 0$, then it is not possible to have three vertices of $S_2$ non-adjacent to $v_i$, for all $i\in\{1,2,3,4\}$. 
\end{claim}

\begin{proof}
Assume the contrary and let $u_1, u_2, u_3 \in S_2$
be non-adjacent to $v_4$. Contract all edges between $S_2 \setminus \{u_1, u_2, u_3\}$ and $\{u_1, u_2, u_3\}$. Also, contract the edges between $S_1$ and 
$\{v_1, v_2, v_3, v_4\}$.  This will create a $K_{4,4}$-minor, a contradiction. 
\end{proof}

\medskip

\begin{claim}\label{claim S2 leq 4}
	If $|S_4|=|S_3| = 0$, then $|S_2| \leq 4$. 
\end{claim}
\begin{proof}
Follows directly from Claims~\ref{S4=S3=0, S2 has 3 with one vi}, 
\ref{S2 have 2 with same parents, not possible} and~\ref{S4=S3=0, S2 has 3 with one vi missing}. 
\end{proof}

\medskip

\begin{lemma}\label{claim S2=4}
	If $\delta(G)=3$ and $\Delta(G)\geq 4$ for a graph $G \in \mathcal{PP}_2$, then the following holds: if ${|S_4|=|S_3| = 0}$ and $|S_2| = 4$, then $G$ is isomorphic to $M^{=}_{11}, M^{-}_{11}$, or $M_{11}$. 
\end{lemma}
    
    \begin{proof}
    Assume that  $S_2 = \{u_1, u_2, u_3, u_4\}$.
    Thus due to Claims~\ref{S4=S3=0, S2 has 3 with one vi}, 
    \ref{S2 have 2 with same parents, not possible} and~\ref{S4=S3=0, S2 has 3 with one vi missing} without loss of we may suppose that 
    $u_i$ is adjacent to $v_i$ and $v_{i+1}$, for all $i \in \{1,2,3,4\}$ and the $+$ operation on the indices is taken modulo $4$. If $S_2$ does not have a perfect matching, then it will force a $K_{4,4}$-minor. Thus we must have the edges $u_1u_3$ and $u_2u_4$. 
    Also $X = \emptyset$, as otherwise there will be a $K_{3,5}$-minor in $G$ (see the partition $\{v,(u_1u_3),(u_2u_4)\}\sqcup \{x,v_1,v_2,v_3,v_4\}$, where $x\in X$). 
    
    \medskip 

    Next, we claim that for all $i\in \{1,2,3,4\}$, $|N(v_i) \cap S_1| \leq 1$. Suppose $|N(v_1)\cap S_1|\geq 2$. Let $w_1,w_2 \in N(v_1)\cap S_1$. Then $w_1$ reaches $v_2$ either via $u_2$ or some vertex in $N(v_2)\cap S_1$, and $w_1$ reaches $v_4$ either via $u_3$ or some vertex in $N(v_4)\cap S_1$. Similarly $w_2$ reaches $v_2$ and $v_4$. Contract the edge $v_2u_2$ and $v_4u_3$ and the edges between $S_1$ and $\{v_2,v_4\}$ to obtain a $K_{3,5}$-minor, a contradiction (see the partition $\{v_1,(v_2u_2),(v_4u_3)\}\sqcup \{w_1,w_2,u_1,u_4,v\}$). Thus $|N(v_1)\cap S_1|\leq 1$. A similar analysis holds for $v_2, v_3$, and $v_4$. Hence, for all $i\in \{1,2,3,4\}$, $|N(v_i) \cap S_1| \leq 1$.

    \medskip

    Next, we claim that $|S_1| \leq 2$. If $|S_1| \geq 3$, then without loss of generality assume that $w_1 \in N(v_1)\cap S_1$. If $w_1$ does not reach $\{v_2,v_3,v_4\}$ via $u_2,u_3$, then it uses vertices from $S_1$, forcing a $K_{4,4}$-minor in $G$. Thus $w_1$ has to use at least one of $u_2,u_3$ to reach $\{v_2,v_3,v_4\}$. Suppose $w_1$ is not adjacent to $u_2$, then it is adjacent to $u_3$. Then $w_1$ reaches $v_2$ via $w_2\in S_1$. Now $w_2$ cannot reach $v_3$ via $u_2$ or $u_3$, else a triangle is induced in $G$. Thus $w_2$ reaches $v_3$ via $w_3\in S_1$. This forms a $K_{4,4}^-$-minor in $G$ (see the partition $\{v,(u_1u_3),(u_2u_4),(w_1w_2w_3)\}\sqcup \{v_1,v_2,v_3,v_4\}$). Next, suppose $w_1$ is not adjacent to $u_3$, then it is adjacent to $u_2$. Then, similarly, a $K_{4,4}^-$-minor is obtained in $G$. Thus $|S_1| \leq 2$. 

    \medskip
    
    Observe that if $|S_1|=0, 1$, or $2$, then $G$ is isomorphic to $M^{=}_{11}, M^{-}_{11}$, or $M_{11}$, respectively. It is easy to observe the cases when $|S_1|=0, 1$. For $|S_1|=2$, without loss of generality we will have two cases: when the two vertices of $S_1$ are adjacent to $v_1$ and $v_2$; and when the two vertices of $S_1$ are adjacent to $v_1$ and $v_3$. In the first case, we get a graph isomorphic to $M_{11}$. In the second case, we get a $K_{4,4}$-minor in $G$. We briefly describe the second case. Let $N(v_1)\cap S_1=\{w_1\}$ and  $N(v_1)\cap S_3=\{w_2\}$. Then $w_1$ reaches $v_2$ via $u_2$, and $v_4$ via $u_3$. And $w_2$ reaches $v_4$ via $u_4$, and $v_2$ via $u_1$. All these edges are forced, or else a triangle is induced in $G$. The only possible way that $w_1$ reaches $w_2$, without inducing a triangle in $G$, is directly by an edge. This forms a $K_{4,4}$-minor in $G$ (see the partition $\{w_1,(vv_3),u_1,u_4)\}\sqcup \{w_2,v_1,(u_3v_4),(u_2v_2)\}$). 
    \end{proof}

    \medskip
    
    \begin{lemma}\label{claim S2=3}
	If $\delta(G)=3$ and $d(v)=\Delta(G)\geq 4$ for a graph $G \in \mathcal{PP}_2$, then the following holds: if ${|S_4|=|S_3| = 0}$ and $|S_2| = 3$, then $G$ is isomorphic to $M^{=}_{11}, M^{-}_{11}$, or $M_{11}$. 
\end{lemma}
    \begin{proof}
    Assume that  $S_2 = \{u_1, u_2, u_3\}$.
    Thus due to Claims~\ref{S4=S3=0, S2 has 3 with one vi}, 
    \ref{S2 have 2 with same parents, not possible}, and~\ref{S4=S3=0, S2 has 3 with one vi missing}, without loss of generality, we may suppose that 
    $u_i$ is adjacent to $v_i$ and $v_{i+1}$, for all $i \in \{1,2,3\}$.
    If $u_1$ is not adjacent to $u_3$, then it will force a $K_{4,4}$-minor. Thus we must have the edge $u_1u_3$. 

    If $N(v_2)\cap S_1 \neq \emptyset$, then let $w_1\in N(v_2)\cap S_1$. Now $w_1$ reaches $v_3$ via $w_2\in S_1$. This forces a $K_{4,4}$-minor (see the partition $\{v,(u_1u_3),u_2,(w_1w_2)\}\sqcup \{v_1,v_2,v_3,v_4\}$). Thus $N(v_2)\cap S_1 = \emptyset$. Similarly, by symmetry, $N(v_3)\cap S_1 = \emptyset$. 
    
    Therefore, every vertex in $S_1$ is adjacent to $u_2$ to reach $v_2$ and $v_3$. Also, every vertex in $N(v_1)\cap S_1$ is adjacent to $u_3$ to reach $v_3$, and every vertex in $N(v_4)\cap S_1$ is adjacent to $u_1$ to reach $v_2$. 

    Next to satisfy $\delta(G)\geq 3$, $N(v_1)\cap S_1$ and $N(v_4)\cap S_1$ have at least one vertex. If $|N(v_1)\cap S_1|\geq 2$, then let $w_1,w_1'\in N(v_1)\cap S_1$ and $w_2\in N(v_4)\cap S_1$. This forces a $K_{3,5}$-minor (see the vertices in the partition $\{(u_1u_3),u_2,(v_1vv_4)\}\sqcup \{w_1,w_1',v_2,v_3,w_2\}$). Thus $N(v_1)\cap S_1=\{w_1\}$ and $N(v_4)\cap S_1=\{w_2\}$. 

    Now let us consider $X$. Suppose $X\neq \emptyset$: let $x\in X$. Both $u_1$ and $u_3$ are not adjacent to $x$, else $|S_3|\geq 1$ which we have dealt earlier. If $u_1$ and $u_3$ do not use $w_2$ and $w_1$, respectively, to reach $x$, then a $K_{3,5}$-minor is forced (see the vertices in the partition $\{(u_1u_3),(w_1u_2w_2),v\}\sqcup \{x,v_1,v_2,v_3,v_4\}$). Hence $x$ is adjacent to $w_1,w_2$. 

    If $|X|\geq 2$, then replacing $v$ by $w_1$ we have $|S_3|\geq 1$ which we have dealt earlier. Hence $|X| \leq 1$. 
    
    Observe that if $|X|=0$ or $1$, then $G$ is $M^{-}_{11}$ or $M_{11}$, respectively. 
    \end{proof}

    
    
    

    \medskip
    
    \begin{lemma}\label{claim S2=2}
	If $\delta(G)=3$ and $d(v)=\Delta(G)\geq 4$ for a graph $G \in \mathcal{PP}_2$, then the following holds: if ${|S_4|=|S_3|=0}$ and $|S_2| = 2$, then $G$ is isomorphic to $K_{3,4}$, $K_{3,4}^*$, or $W_8^{+}$.
\end{lemma}
    
    \begin{proof}
    Assume that  $S_2 = \{u_1, u_2\}$.
    Thus due to Claims~\ref{S4=S3=0, S2 has 3 with one vi}, 
    \ref{S2 have 2 with same parents, not possible} 
    and~\ref{S4=S3=0, S2 has 3 with one vi missing} 
    without loss of generality we may suppose one of the two scenarios: 
    (i) $u_1$ is adjacent to $v_1, v_2$, and $u_2$ is adjacent to $v_2$ and $v_3$,
    (ii) $u_1$ is adjacent to $v_1, v_2$, and $u_2$ is adjacent to $v_3$ and $v_4$.

    \medskip
    
\textit{Case (i):} First suppose $u_1$ is adjacent to $v_1, v_2$, and $u_2$ is adjacent to $v_2$ and $v_3$. 
Observe that $v_2$ cannot be adjacent to any vertex of $S_1$, 
as otherwise a $K_{4,4}$-minor will be created. 
Also $|N(v_4) \cap S_1| \geq 2$ as $\delta(G) \geq 3$. 
Every vertex in $N(v_1)\cap S_1$ reaches $v_2$ via $u_2$, and every vertex in $N(v_3)\cap S_1$ reaches $v_2$ via $u_1$. 

Now $u_1$ reaches $v_4$ via $w_4\in N(v_4)\cap S_1$. $w_4$ cannot reach $v_3$ via any vertex in $N(v_3)\cap S_1$, as a triangle is induced. Thus $w_4$ reaches $v_3$ via $u_2$. 

Now $w_4'\in (N(v_4)\cap S_1)\setminus \{w_4\}$ reaches $v_2$ via $u_1$ or $u_2$. Next, $w_4'$ cannot use $N(v_3)\cap S_1$ or $N(v_1)\cap S_1$ to reach the yet unreached vertex in $\{v_1,v_3\}$, else a triangle is induced. 
Thus $w_4'$ is adjacent to both $u_1$ and $u_2$. Similarly, every vertex in $N(v_4)\cap S_1$ is adjacent to $u_1$ and $u_2$. 

Thus any vertex in $N(v_1)\cap S_1$ cannot reach any vertex in $N(v_3)\cap S_1$ via any vertex of $N(v_4)\cap S_1$. Hence every vertex in $N(v_1)\cap S_1$ is adjacent to every vertex in $N(v_3)\cap S_1$.  
This will reduce the case to the case of $|S_3|+|S_4| \geq 1$ where $u_1$ plays the role of $v$ (note that $u_1$ and $u_2$ share at least three common neighbors).  
    
    \medskip
    
    \textit{Case (ii):}  Next assume that $u_1$ is adjacent to $v_1, v_2$, and $u_2$ is adjacent to $v_3$ and $v_4$. Here $u_1$ is adjacent to $u_2$, as otherwise a $K_{4,4}$-minor will be created. 
    
    Note that $|N(v_i) \cap S_1| \geq 1$ as $\delta(G) \geq 3$. 
    All vertices of $N(v_1)\cap S_1$ and $N(v_2)\cap S_1$ cannot be adjacent to $u_2$ as a vertex in $N(v_1)\cap S_1$ reaches $v_2$ via a vertex in $N(v_2)\cap S_1$. Without loss of generality assume that $w_1\in N(v_1)\cap S_1$ is not adjacent to $u_2$. 
    If $x\in X$, then a $K_{3,5}$-minor will be created (see the partition $\{v, w_1, (u_1u_2)\}\sqcup \{x,v_1,v_2,v_3,v_4\}$). Hence $X=\emptyset$. 

    Next, we claim that $|N(v_i) \cap S_1| \leq 1$. Suppose $w_1,w_1'\in N(v_1)\cap S_1$. Both of them cannot be adjacent to $u_2$, else this case reduces to $|S_3|+|S_4|\geq 1$ where $v_1$ plays the role of $v$ (note that $v_1$ and $u_2$ have at least three neighbors in common). Without loss of generality assume that $w_1\in N(v_1)\cap S_1$ is not adjacent to $u_2$. Thus it reaches $v_2,v_3,v_4$ via vertices in $S_1$. If $w_1'$ also reaches $v_2,v_3,v_4$ via vertices in $S_1$, then a $K_{4,4}$-minor will be created (see the partition $\{v,w_1,w_1',(u_1u_2)\}\sqcup \{v_1,v_2,v_3,v_4\}$). If $w_1'$ reaches $v_3,v_4$ via $u_2$, then let $w_1'$ reach $v_2$ via $w_2$. Now $w_2$ cannot be adjacent to $u_2$, else a triangle is induced. Thus $w_2$ reaches $v_3,v_4$ via vertices in $S_1$. This forces a $K_{4,4}$-minor (see the partition $\{v,w_1,(w_1'w_2),(u_1u_2)\}\sqcup \{v_1,v_2,v_3,v_4\}$). Hence $|N(v_i) \cap S_1| \leq 1$; and since $|N(v_i) \cap S_1| \geq 1$ as $\delta(G) \geq 3$, we have $|N(v_i) \cap S_1| = 1$. 

    Suppose $N(v_i) \cap S_1 = \{w_i\}$, for all $i \in \{1,2,3,4\}$. Then $w_1$ reaches $v_2$ via $w_2$, and $w_3$ reaches $v_4$ via $w_4$. Thus $w_1w_2$ and $w_3w_4$ are edges in $G$. Moreover, due to symmetry, 
    without loss of generality, we may assume the edges $w_2u_2, w_3u_1$ as forced as well.

    This reduces the case to Case (i) of this proof where $u_1$ plays the role of $v$. 
    \end{proof}
    
    \medskip

    \begin{claim}\label{claim S2=1}
	If $|S_4|=|S_3| = 0$, then it is not possible to have $|S_2|=1$. 
\end{claim}
    
    \begin{proof}
    Assume that  $S_2 = \{u_1\}$ and $u_1$ is adjacent to $v_1, v_2$. Note that as $|S_2|=1$, $u_1$ must reach 
    $v_3, v_4$ via (say) $w_{31}, w_{41} \in S_1$, respectively.
    
    Notice that $u_1$ and $v_i$ cannot have two common neighbors from $S_1$, as otherwise the case will be reduced to $|S_2| \geq 2$ where $u_1$ plays the role of $v$. 
    
    Furthermore as $\delta(G) \geq 3$, $v_3$ must have another neighbor $w_{32} \in S_1$. As $w_{32}$ cannot be adjacent to $u_1$, it must reach $v_1, v_2$ via 
    $w_{11}, w_{21} \in S_1$, respectively. Observe that 
    $w_{11}$ is not adjacent to $w_{21}$ in order to avoid 
    creating a triangle. Thus $w_{21}$ must reach $v_1$ via some $w_{12} \in S_1$. 
    
    Now contract all the edges between 
    $S_1$ and $\{v_2, v_3, v_4\}$. This will create a $K_{4,4,}$-minor, a contradiction (see the partition $\{v,u_1,w_{11},w_{12}\}\sqcup \{v_1,v_2,v_3,v_4\}$). 
    \end{proof}
    
    This concludes the case when we have $|S_4|+|S_3|=0$. We will present the summary of it in the following lemma.

\begin{lemma}\label{lem case s3+s4=0}
	If $\delta(G)=3$ and $d(v)=\Delta(G)\geq 4$ for a graph $G \in \mathcal{PP}_2$, then the following holds: if ${|S_4|+|S_3|=0}$, then $G$ is isomorphic to $K_{3,4}$ $K_{3,4}^{*}$, $W_8^{+}$, $M_{11}^{=}$, $M_{11}^{-}$, or $M_{11}$. 
\end{lemma}
    
    \begin{proof}
    Follows directly from Lemmas~\ref{claim S2=4}, \ref{claim S2=3}, and \ref{claim S2=2}, and Claims~\ref{S4=S3=S2=0 not possible} and \ref{claim S2=1}. 
    \end{proof}
    
    \medskip
    
    Finally, we are ready to prove Lemma~\ref{lm:max-deg-4}. 
    \medskip

    \noindent \textit{Proof of Lemma~\ref{lm:max-deg-4}.}
    The result readily follows from Lemmas~\ref{lem case s3+s4=2}, \ref{lem case s3+s4=1}, and~\ref{lem case s3+s4=0}. 
    \qed

    \subsection{Concluding the proof of Theorem~\ref{th plesnik}}
    At last, we can conclude the proof of Theorem~\ref{th t-free ppg}. 
    
    \medskip
    
    \noindent \textit{Proof of Theorem~\ref{th t-free ppg}.}
    The result readily follows from Lemmas~\ref{lm:min-deg-1}, \ref{lm:min-deg-2}, \ref{lm:3-regular}, and~\ref{lm:max-deg-4}.  
    \qed

\section{Direct implications}\label{sec implications}
In Theorem~\ref{th t-free ppg}, we have characterized all triangle-free projective-planar graphs having diameter $2$. 
This has an immediate theoretical implication in the theory of graph homomorphisms of colored mixed graphs, signed graphs, and oriented graphs. We are going to discuss them here.

First let us start with colored mixed graphs which were introduced by Ne\v{s}et\v{r}il and Raspaud~\cite{nesetriljctb}. 
An \textit{$(m,n)$-colored mixed graph} $G$ is a 
graph having 
$m$ different types of arcs and 
$n$ different types of edges. 
Moreover, \textit{colored homomorphism} from
an $(m,n)$-colored mixed graph $G$ to 
another $(m,n)$-colored mixed graph $H$ is 
a vertex mapping $f: V(G) \rightarrow V(H)$ such that 
for any arc (resp., edge) $uv$ of $G$, the induced image
$f(u)f(v)$ is also an arc (resp., edge) of the same type in $H$. 
Observe that for $(m,n) = (0,1), (1,0), (0,2)$,  and $(0,k)$
the study of colored homomorphism of $(m,n)$-colored mixed graphs is the same as studying homomorphisms of 
undirected graphs~\cite{nesetrilbook}, oriented graphs~\cite{ericsurvey}, $2$-edge-colored graphs~\cite{pascalalexsen}, and $k$-edge-colored graphs~\cite{alonmarshall}, respectively. Each of these is a well-studied topic. 

Generalizing the notion of oriented absolute cliques and oriented absolute clique number\footnote{The same is also known as \textit{oriented cliques} or \textit{ocliques} and \textit{oriented clique number} or \textit{oclique number}.}~\cite{klostermeyer}, Bensmail, Duffy, and Sen~\cite{bensmailduffysen} introduced the notion of  $(m, n)$-clique and $(m,n)$-absolute clique number. 
An \textit{$(m, n)$-clique} $C$ is an $(m, n)$-colored mixed graph that does not admit a colored homomorphism to any other $(m,n)$-colored mixed graph having strictly fewer vertices. 
Given a family 
$\mathcal{F}$ of $(m,n)$-colored mixed graphs,
$$\omega_{a(m,n)}(\mathcal{F}) = \max\{|V(G)|:G \in \mathcal{F} \text{ is an } (m,n)\text{-clique}\}.$$

A handy characterization 
of an $(m, n)$-clique is proved by Bensmail, Duffy, and Sen~\cite{bensmailduffysen}. 

\begin{proposition}[\cite{bensmailduffysen}]\label{prop clique1}
An $(m,n)$-colored mixed graph $C$ is an $(m,n)$-clique 
if and only if every pair of non-adjacent vertices $u,w$ of $C$ are connected by a  $2$-path $uvw$ of one of the following types: 
\begin{itemize}
\item[(i)] $uv$ and $vw$ are edges of different colors,

\item[(ii)] $uv$ and $vw$ are arcs (possibly of the same color),

\item[(iii)] $vu$ and $wv$ are arcs (possibly of the same color),

\item[(iv)] $uv$ and $wv$ are arcs of different colors,

\item[(v)] $vu$ and $vw$ are arcs of different colors,

\item[(vi)] exactly one of $uv$ and $vw$ is an edge.
\end{itemize}

\end{proposition}

A $2$-path in an $(m,n)$-graph is a \textit{special $2$-path} if it is 
one among the six types of path listed in Proposition~\ref{prop clique1}. 
If a $2$-path $uvw$ is a special $2$-path, then we say that \textit{$u$ sees $w$ via $v$} and that $u$ and $w$ \textit{disagrees} on $v$. 
If $uvw$ is not a special $2$-path, then we say that $u$  and $w$ \textit{agrees} on $v$.
Due to the above proposition, we know that any underlying graph of an $(m,n)$-clique must have a diameter of at most $2$. 
Moreover, the underlying graph of an $(m,n)$-clique is called an 
\textit{underlying $(m,n)$-clique}. 

\begin{observation}\label{obs und clique diam2}
	An underlying $(m,n)$-clique has a diameter at most $2$. 
\end{observation}

Theorem~\ref{th t-free ppg} and Proposition~\ref{prop clique1} directly imply the following.

\begin{theorem}
For the family $\mathcal{PP}_2$ of $(m,n)$-colored mixed triangle-free projective-planar graphs 
\begin{enumerate}
    \item[(i)] $\omega_{a(1,0)}(\mathcal{PP}_2)= 9 $
    
    \item[(ii)] $\omega_{a(0,2)}(\mathcal{PP}_2)= 8 $
    
    \item[(iii)] $\omega_{a(m,n)}(\mathcal{PP}_2)= (2m+n)^2 +2, 
                                      \text{ for all } 2m+n \geq 3$.
\end{enumerate}
\end{theorem}

\begin{proof}
	(i) Observe that  the graph $W_8^{+}$ is an underlying $(1,0)$-clique on $9$ vertices (see Fig.~\ref{fig mixed cliques}(a) for the relevant instance). 
	This implies $\omega_{a(1,0)}(\mathcal{PP}_2) \geq 9$.

 \begin{figure}
    \centering
    \begin{tabular}{cc}

         \includegraphics[page=4]{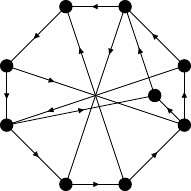}
    
 & \includegraphics[page=3]{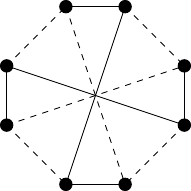}
     \\
        (a) & (b) 
    \end{tabular}
    \caption{(a) A $(1,0)$-clique on $9$ vertices. (b) A $(0,2)$-clique on $8$ vertices.}
    \label{fig mixed cliques}
\end{figure}


	Let $G^*$ be a triangle-free projective planar $(1,0)$-clique having at least $9$ vertices. Thus, by Observation~\ref{obs und clique diam2} its underlying graph, say $G$, must belong to $\mathcal{PP}_2$. 
	We are going to show that such a $G^*$ does not exist. 
	Due to Theorem~\ref{th t-free ppg}, 
	it is enough to restrict ourselves 
	to checking whether any graph listed in the theorem can be $G$ or not.

	As all Plesn\'ik graphs are triangle-free planar graphs, and as it is known~\cite{CHAKRABORTY202329} that the largest triangle-free planar underlying $(1,0)$-clique has six vertices. 
	Moreover, $K_{3,3}, K_{3,4}, W_8, W_8^{+}, M_{11}^{=}$, and $K_{3,4}^{*}$ has less or equal to nine vertices. Furthermore, using Proposition~\ref{prop clique1} it is possible to verify that $P_{10}, M_{11}$, and $M_{11}^{-}$ are not underlying $(1,0)$-cliques. Hence, we are only left with verifying whether it is possible to have $K_{3,4}(t)$ or $K_{3,3}(t)$ as $G$ or not, where $t \geq 3$.

	Let us suppose that $G$ is either  $K_{3,4}(t)$ or $K_{3,3}(t)$ for some $t \geq 2$, and $a_1, a_2, \ldots, a_t$ are its vertices of degree two. 
	Notice that, all the $a_i$s are adjacent to exactly two vertices, 
	say $b_1, b_2$. Note that, there are vertices $b_1'$ and $b_2'$ such that
	$b_j'$ is non-adjacent to $a_i$ and the only $2$-path connecting $b_j'$ to $a_i$ is $b_j$, for all $j \in \{1,2\}$. 
	Thus in $G^*$, all $a_i$s must see $b_j'$ via $b_j$, which implies that all $a_i$s must agree with each other on both $b_1$ and $b_2$. Thus, $a_1$ is neither adjacent to $a_2$ nor there is a special $ 2$ path among them. 
	Hence $G$ cannot be an underlying $(1,0)$-clique if it is either  $K_{3,4}(t)$ or $K_{3,3}(t)$ for some $t \geq 2$.

	\bigskip

	(ii) Observe that  the graph $W_8$ is an underlying $(0,2)$-clique on $8$ vertices (see Fig.~\ref{fig mixed cliques}(b) for the relevant instance). 
	This implies $\omega_{a(0,2)}(\mathcal{PP}_2) \geq 8$. 
	The proof of the upper bound can be done similarly to the proof of (i).

	\bigskip

	(iii) It is known~\cite{CHAKRABORTY202329} that $\omega_{a(m,n)}(\mathcal{P}')= (2m+n)^2 +2$,  for all $2m+n \geq 3$
	for the family of 
	triangle-free planar graphs.  
	This implies  $\omega_{a(m,n)}(\mathcal{PP}_2)= (2m+n)^2 +2$, 
	 for all  $2m+n \geq 3$.
	
	For the upper bound, as $(2m+n)^2 +2 \geq 11$ for all $2m+n \geq 3$, and as every non-planar graphs except $K_{3,4}(t)$ and $K_{3,3}(t)$ for $t \geq (2m+n)^2 - 4$ from the graphs listed in Theorem~\ref{th t-free ppg} has less than or equal to $11$ vertices, it is enough to show that  $K_{3,4}(t)$ and $K_{3,3}(t)$
	are not underlying $(m,n)$-cliques for $t \geq (2m+n)^2 - 4$. 
	
	Let $G^*$ be an $(m,n)$-clique having at least $(2m+n)^2+3$ vertices such that its underlying graph $G$ is 
	$K_{3,4}(t)$ and $K_{3,3}(t)$ for some $t \geq (2m+n)^2 - 4$. 
	Let $a_1, a_2, \ldots, a_t$ be the vertices of degree two of $G$. 
	Notice that, all the $a_i$s are adjacent to exactly two vertices, 
	say $b_1, b_2$. 
	Note that, there are vertices $b_1'$ and $b_2'$ such that
	$b_j'$ is non-adjacent to $a_i$ and the only $2$-path connecting $b_j'$ to $a_i$ is $b_j$, for all $j \in \{1,2\}$. 
	Thus in $G^*$, all $a_i$s must see $b_j'$ via $b_j$. This implies that the adjacency between $a_i$ and $b_j$ is different from the adjacency between $b_j'$ and $b_j$, for each $i$ and $j$. 
	This implies that the adjacency between $a_i$ and $b_j$ can be one of the $(2m+n)-1$ types (each type of arc gives two adjacency options due to directions, and each type of edge gives one adjacency option). As $a_i$ also must see each other via $b_1$ or $b_2$, the number of $a_i$s is bounded by 
	$$t \leq (2m+n-1)^2 = (2m+n)^2 - 2(2m+n) +1 \leq (2m+n)^2 -5$$
	for all $(2m+n) \geq 3$. However, this is a contradiction as we assumed $t \geq (2m+n)^2 - 4$.
\end{proof}

Now we turn our focus towards variants of colored homomorphism of 
$(0,2)$ and $(1,0)$-colored mixed graphs (that is, $2$-edge-colored graphs and oriented graphs). The variants are known as homomorphisms of signed graphs~\cite{rezasigned} and pushable homomorphisms of oriented graphs~\cite{klostermeyer04}, respectively. 

Homomorphisms of signed graphs were introduced by Naserasr, Rollov\'a, and Sopena~\cite{rezasigned} who also defined and characterized signed absolute clique and signed absolute clique numbers. Naserasr, Rollov\'a, and Sopena~\cite{rezasigned} also showed how using the notion of 
homomorphism of signed graph one can capture, as well as extend, many of the classical graph theory results and conjectures including the Four-Color Theorem and Hadwiger's Conjecture~\cite{hadwiger}. It motivated a number of research works and generated a lot of interest within 
a short span of time~\cite{florentjctb, florentdam, dasprabhusen, rezasigned, pascalalexsen}.  In order to avoid a long series of definitions, we would like to define the notion using its equivalent characterization. 

A \textit{signed graph} $(G, \Sigma)$ is a graph with either positive
or negative sign assigned to its edges. 
A \textit{signed absolute clique} $(C, \Lambda)$ 
is a signed graph whose any two non-adjacent vertices are part of a $ 4$ cycle having an odd number of negative edges. 
Given a family 
$\mathcal{F}$ of signed graphs,
$$\omega_{as}(\mathcal{F}) = \max\{|V(G)|: 
G \in \mathcal{F} \text{ and } (G, \Sigma) \text{ is a  signed absolute clique}\}.$$

On the other hand, pushable homomorphism of 
oriented graphs were introduced by Klostermeyer and MacGillivray~\cite{klostermeyer} which motivated some further research works 
on that topic. These  include  one work 
due to Bensmail, Nandi, and Sen~\cite{bensmailnandisen} that introduced and characterized the notion of pushable absolute clique of oriented graphs.   In order to avoid a long series of definitions, we would like define the notion using its equivalent characterization. 

An \textit{oriented graph} $\overrightarrow{G}$ is a directed graph without any directed cycle of length $1$ or $2$.  
A \textit{pushable absolute clique} $\overrightarrow{C}$ 
is an oriented graph whose two non-adjacent vertices are part of a $ 4$ cycle having an odd number of arcs in a clockwise direction.  
Given a family 
$\mathcal{F}$ of oriented graphs,
$$\omega_{ap}(\mathcal{F}) = \max\{|V(\overrightarrow{G})|: 
\overrightarrow{G} \in \mathcal{F} \text{ is a pushable absolute clique}\}.$$

Thus from Theorem~\ref{th t-free ppg} and  the two characterizations (definitions), we have the following theorem.

\begin{theorem}
For the families $\mathcal{PP}_2$ (resp.  $\mathcal{PP}_2$) of oriented (resp. signed) triangle-free projective-planar graphs 
$$\omega_{ap}(\mathcal{PP}_2)=\omega_{as}(\mathcal{PP}_2)= 7.$$
\end{theorem}

\begin{proof}
Observe that  there exist a signed absolute clique and a pushable absolute clique having $K_{3,4}$ as their underlying graphs  (see Fig.~\ref{fig push-signed cliques} for the relevant instances). 
	This implies $\omega_{as}(\mathcal{PP}_2) \geq 7$ and 
	$\omega_{ap}(\mathcal{PP}_2) \geq 7$. 

\begin{figure}
    \centering
    \begin{tabular}{cc}

         \includegraphics[page=1]{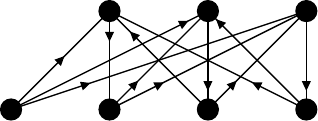}
    
 & \includegraphics[page=2]{ppg_signed_push.pdf}
     \\
        (a) & (b) 
    \end{tabular}
    \caption{(a) A pushable absolute clique on $7$ vertices. (b) A signed absolute clique on $7$ vertices.}
    \label{fig push-signed cliques}
\end{figure}	


For the upper bound, note that if $G$ is the underlying graph of a pushable absolute clique or a signed  absolute clique, then it must have the following property: any two non-adjacent vertices of $G$ must be connected 
by two internally disjoint $2$-paths. Observe that, among the graphs listed in Theorem~\ref{th t-free ppg}, the only graphs that have this property are $K_{3,3}$, $K_{3,4}$, and $K_{2,t}$ for $t \geq 2$. 
However, notice that a pushable absolute clique (resp., signed absolute clique) is, in particular, a $(1,0)$-clique (resp., $(0,2)$-clique). Moreover, the $(1,0)$-absolute clique number 
(resp., $(0,2)$-absolute clique number) for the family of triangle-free planar graphs is at most $6$~\cite{CHAKRABORTY202329}. As $K_{2,t}$ is a triangle-free planar graph, for all $t \geq 2$, we are done. 
\end{proof}

\section{Conclusions}\label{sec conclusions}
In this paper, we gave a characterisation of triangle-free  projective planar graphs of diameter $2$ and proved that the domination number of this class of graphs is at most $3$. Moreover, there are only seven triangle-free projective planar graphs with a diameter 2 for which the equality holds. This raises a natural question.

\begin{question}\label{Q type1}
Given a surface $\mathbb{S}$, can you find a tight upper bound on the domination number of triangle-free graphs with diameter $2$ that can be embedded on $\mathbb{S}$?
\end{question}

Goddard and Henning~\cite{goddard2002} proved that for any 
integer $g \geq 0$ the number of graphs with orientable genus $g$, diameter $2$, and domination number greater than $2$ is finite. 
For the $g=0$ case, that is for planar graphs, they reported~\cite{goddard2002} that there exists only one planar graph of diameter two and domination number greater than two. However, the analogous problem for higher genus surfaces is still unsolved. 
This motivates the following question.

\begin{question}\label{Q type2}
Given a positive integer $g$, how many triangle-free graphs of genus $g$ and diameter $2$ are there with a domination number greater than or equal to $3$?
\end{question}

We observed that the maximum order of a triangle-free projective planar graph with diameter $2$ and domination number $3$ is $11$. This motivates the following.

\begin{question}\label{Q type3}
Given a positive integer $g$, 
what is the highest order of a triangle-free graph of genus $g$ and diameter $2$ with a domination number greater than $3$?
\end{question}

\section{Acknowledgements}
We thank Qianping Gu for providing the code to find embeddings of projective-planar graphs~\cite{yu2014}. We thank all the anonymous reviewers who have gone through the draft multiple times minutely and their suggestions which have greatly enhanced the presentation and the readability of the article.





\bibliography{reference}

\end{document}